\documentclass[12pt]{amsart}
\usepackage[T2A]{fontenc}
\usepackage[english]{babel}
\usepackage{amsaddr}
\usepackage{amsmath}
\usepackage{amssymb}
\usepackage{amsfonts}
\usepackage{epsfig}
\usepackage{srcltx}
\usepackage{subfigure}
\usepackage{cite}
\usepackage{float}
\usepackage{mathtools}
\usepackage[a4paper, mag=1000, includefoot, left=2cm, right=2cm, top=2cm, bottom=2cm, headsep=1cm, footskip=1cm]{geometry}
\usepackage[unicode,colorlinks]{hyperref}
\hypersetup{colorlinks=true, citecolor=blue, linkcolor=blue}

\newtheorem{Th}{Theorem}
\newtheorem{Lem}{Lemma}

\begin{document}

\thispagestyle{empty}

\title[]{Resonances in nonlinear systems with a decaying chirped-frequency excitation and noise}

\author[O.A. Sultanov]{Oskar A. Sultanov}

\address{
Institute of Mathematics, Ufa Federal Research Centre, Russian Academy of Sciences, Chernyshevsky street, 112, Ufa 450008 Russia;\\
Chebyshev Laboratory, St. Petersburg State University,
14th Line V.O., 29, Saint Petersburg 199178 Russia
}
\email{oasultanov@gmail.com}


\maketitle

{\small
\begin{quote}
\noindent{\bf Abstract.} 
The influence of multiplicative white noise on the resonance capture of strongly nonlinear oscillatory systems under chirped-frequency excitations is investigated. It is assumed that the intensity of the perturbation decays polynomially with time, and its frequency grows according to a power low. Resonant solutions with a growing amplitude and phase, synchronized with the excitation, are considered. The persistence of such a regime in the presence of stochastic perturbations is discussed. In particular, conditions are described that guarantee the stochastic stability of the resonant modes on infinite or asymptotically large time intervals. The technique used is based on a combination of the averaging method, stability analysis and construction of stochastic Lyapunov functions. The proposed theory is applied to the Duffing oscillator with a chirped-frequency excitation and noise.

 \medskip



\end{quote}
}
{\small

\section*{Introduction}
The present paper is devoted to the study of resonant phenomena in nonlinear pendulum-type systems. Resonance is usually referred to as the effect of increasing the energy of a system under the influence of the oscillating driving. This phenomenon is widely applicable in various fields: from acoustics and optics to electronics and mechanics (see, for example,~\cite{RS16,MKSS18}). The persistent amplification of the amplitude occurs if the frequency of the driving is close to the natural frequency of the system. In the case of non-isochronous systems, perturbations with a chirped frequency turn out to be very effective~\cite{LF09,GMetal17,LF19etal}. In this case, the excitation frequency changes continuously with time. Resonant capture under small, slowly varying excitation with a chirped frequency has been studied in~\cite{LK08,OK19,AK20,OS21}. In this paper, the presence of a small parameter is not assumed and a special class of chirped perturbations with a decaying intensity is considered. 

Note that the influence of decaying perturbations on autonomous systems have been studied in many papers. In some cases, such perturbations can not disrupt the autonomous dynamics~\cite{RB53,LM56,LDP74}. However, in the general case this cannot be guaranteed: the global properties of solutions to perturbed and unperturbed systems can differ significantly (see, for example,~\cite{HRT94,OS21IJBC}). Bifurcations in asymptotically autonomous systems have been studied in~\cite{LRS02,KS05,MR08,OS22Non}. The effect of decaying oscillatory perturbations with an asymptotically constant frequency on nonlinear autonomous systems in the plane was studied in~\cite{OS21DCDS,OS21JMS}, where the dynamics near the equilibrium in resonance and non-resonance cases was discussed. The long-term asymptotic behaviour for solutions of similar but linear systems were investigated in~\cite{PN06,BN10}. The decaying perturbations with chirped frequency were considered in~\cite{OS23DCDSB,OS22JMS}, where the conditions for the resonant capture far from the equilibrium were described. In this paper, we study the stability of such resonant regimes with respect to multiplicative stochastic perturbations of white noise type. 

It is well known that even small stochastic perturbations can cause the trajectories of dynamical systems to leave any bounded domain~\cite{FW98}. Many papers (see, for example,~\cite{AMR08,SS10,ZWG12,BRR15}) have studied the effect of autonomous stochastic perturbations on dynamical systems. Fading stochastic perturbations of scalar autonomous systems were considered in~\cite{AGR09,ACR11,KT13}. Bifurcations and asymptotic regimes for solutions in the vicinity of the equilibrium of near-Hamiltonian systems with decaying noise were discussed in~\cite{OS22IJBC,OS23SIAM}. However, the influence of stochastic perturbations on the resonant capture far from the equilibrium has not been studied earlier. This is the subject of the present paper.

The paper is organized as follows. In Section~\ref{Sec1}, the statement of the problem is given, the class of perturbations is described, and a motivating example is presented. Preliminaries and auxiliary results that will be used below are contained in Section~\ref{Sec2}. The main results are presented in Section~\ref{MR}. The justification is provided in the subsequent sections. In particular, in Section~\ref{Sec4} we prove the auxiliary results on the properties of a truncated system in the amplitude-angle variables. In Section~\ref{Sec5} we construct the averaging transformation that simplifies the perturbed system in the first asymptotic terms at infinity in time. Then in Section~\ref{Sec6}, we consider a truncated system obtained from the simplified one by dropping the diffusion terms, and discuss the existence and stability of different long-term asymptotic regimes. In Section~\ref{Sec7}, we prove the persistence of the resonant mode in the full system by constructing suitable stochastic Lyapunov functions. In Section~\ref{Sex}, the proposed theory is applied to the Duffing oscillator with various chirped perturbations and multiplicative noise. The paper concludes with a brief discussion of the results obtained.

\section{Problem statement}\label{Sec1}
Consider the system of It\^{o} stochastic differential equations
\begin{gather}\label{ps}
d\begin{pmatrix} x_1\\ x_2\end{pmatrix} = 
\begin{pmatrix} x_2\\ -U'(x_1)+t^{-\alpha} Q\left(x_1,x_2,S(t)\right)\end{pmatrix}dt + \begin{pmatrix} 0\\ t^{-\gamma}\mu \sigma(x_1,x_2,S(t))\end{pmatrix}dw(t)
\end{gather}
as $t\geq t_0\geq 1$ with $\alpha,\gamma,\mu\in\mathbb R_+$ and a polynomial potential 
\begin{gather}\label{Ucond}
	U(x)\equiv \frac{x^{2h+2}}{2h+2}+\sum_{i=0}^{2h+1}u_i x^i, \quad u_i\in\mathbb R, \quad h\in\mathbb Z_+, 
\end{gather}
where $w(t)$ is a Weiner process on a probability space $(\Omega,\mathfrak A,\mathbb P)$ and the function $S(t)\equiv s t^{\beta+1}$ with $\beta,s\in\mathbb R_+$ corresponds to the phase of the excitation. It is assumed that the functions $Q(x_1,x_2,S)$ and $\sigma(x_1,x_2,S)$, defined for all $(x_1,x_2,S)\in\mathbb R^3$, are infinitely differentiable, do not depend on $\omega\in\Omega$ and have the following form:
\begin{gather}\label{Qform}
\begin{split}
& Q(x_1,x_2,S)\equiv \sum_{i=0}^p\sum_{j=0}^l Q_{i,j}(S) x_1^i x_2^j, \quad l,p\in\mathbb Z, \quad 0\leq l\leq p\leq 2h+1, \\
& \sigma(x_1,x_2,S)\equiv \sum_{i=0}^n\sum_{j=0}^m \sigma_{i,j}(S) x_1^i x_2^j, \quad n,m\in\mathbb Z, \quad 0\leq n\leq p, \quad 0\leq m\leq l,
\end{split}
\end{gather}
where the coefficients $Q_{i,j}(S)$ and $\sigma_{i,j}(S)$ are $2\pi$-periodic with respect to $S$. Thus, \eqref{ps} can be viewed as a pendulum-type system with a chirped-frequency excitation and multiplicative noise. The parameter $\mu$ with the decaying function $t^{-\gamma}$ are responsible for the noise intensity. 
 
From~\cite{OS23DCDSB} it follows that if $\mu=0$, then the deterministic system \eqref{ps} may have a family of solutions with unlimitedly growing amplitude $|x_1(t)|+|x_2(t)|\to \infty$ as $t\to\infty$. This behaviour is associated with a resonance capture phenomenon. In this paper, we assume that $\mu\neq 0$ and study the effect of a noise on the stability of such solutions. 

As an example, consider the perturbed Duffing oscillator
\begin{gather}\label{Ex}
dx_1=x_2\,dt, \quad dx_2=\left(x_1-x_1^3+t^{-\alpha}  \mathcal Q x_1^p \cos S(t)\right) dt+ t^{-\gamma} \mu x_1^n \cos S(t)\, dw(t),
\end{gather}
with $\mathcal Q ={\hbox{\rm const}}$. It can easily be checked that system \eqref{Ex} has the form of \eqref{ps} with $U(x)\equiv x^4/4-x^2/2$, $h=1$, $l=m=0$, $Q(x_1,x_2,S)\equiv  \mathcal Q   x_1^p \cos S$ and $\sigma(x_1,x_2,S)\equiv x_1^n \cos S$. Note that all trajectories of the corresponding limiting autonomous system 
\begin{gather*}
\frac{dx_1}{dt}=x_2, \quad \frac{dx_2}{dt}=x_1-x_1^3
\end{gather*}
 are bounded (see~Fig.~\ref{pp}). Moreover, the solutions with $H(x_1(t),x_2(t))\equiv E>0$, where $H(x_1,x_2)\equiv x_1^4/4+(x_2^2-x_1^2)/2$, correspond to non-isochronous oscillations with a period $\tilde T(E)=\mathcal O(E^{-1/4})$ as $E\to\infty$. Numerical analysis of system \eqref{Ex} with $\mu=0$ shows that under certain conditions on the parameters of the chirped perturbation, the resonant solutions with an unboundedly growing amplitude $\rho(t)=[H(x_1(t),x_2(t))]^{1/4}$ take place (see~Fig.~\ref{fex0}, a, black curves). There are also non-resonant solutions with limited amplitude (see~Fig.~\ref{fex0}, a, grey curves). If $\mu\neq 0$, the stochastic perturbations may disrupt the resonant capture (see~Fig.~\ref{fex0}, b). 

Thus, our goal is to find conditions under which the capture into the resonance persists in the stochastic system \eqref{ps} with a high probability.

\begin{figure}
\centering
\includegraphics[width=0.4\linewidth]{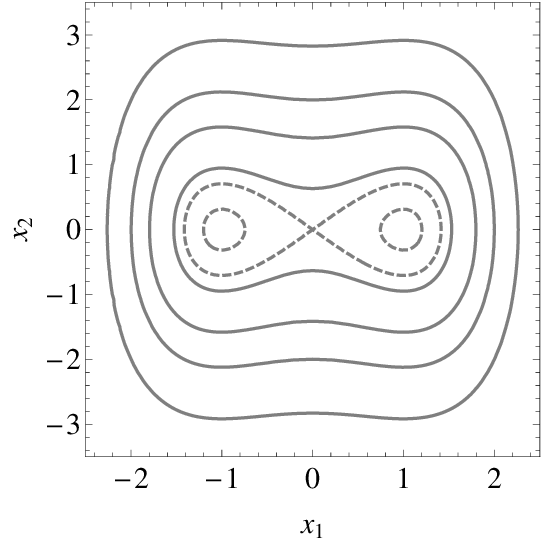}
\caption{\footnotesize The level lines of $H(x_1,x_2)\equiv x_1^4/4+(x_2^2-x_1^2)/2$. The solid curves correspond to $H(x_1,x_2)>0$.} \label{pp}
\end{figure}

\begin{figure}
\centering
\subfigure[ $\mu=0$] {\includegraphics[width=0.4\linewidth]{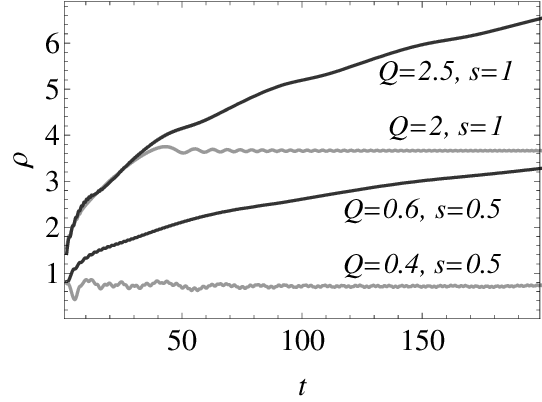}}
\hspace{2ex}
 \subfigure[$\mu=0.2$, $\mathcal Q=2.5$, $s=1$]{\includegraphics[width=0.4\linewidth]{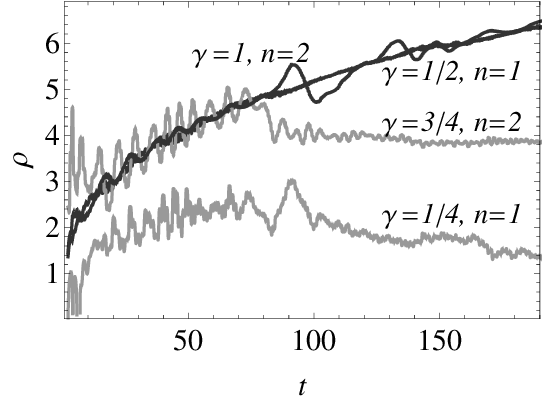}}
\caption{\small The evolution of $\rho(t)\equiv [H(x_1(t),x_2(t))]^{1/4}$ for sample paths of solutions to system \eqref{Ex} with $\alpha=\beta=1/3$ and $p=0$.} \label{fex0}
\end{figure}

\section{Preliminaries and auxiliary results}
\label{Sec2}
Consider the limiting system
\begin{gather}\label{limsys}
\frac{dx_1}{dt}=\partial_{x_2} H(x_1,x_2), \quad \frac{dx_2}{dt}=-\partial_{x_1} H(x_1,x_2), \quad H(x_1,x_2)\equiv \frac{x_2^2}{2}+U(x_1).
\end{gather}
It can easily be checked that there exists $\rho_0 > 0$ such that for every $\rho\geq \rho_0$ the level lines $\{(x_1, x_2) \in\mathbb  R^2 : H(x_1, x_2) = \rho^{2h+2}\}$ determine closed curves on the phase plane $(x_1, x_2)$ and correspond to $T(\rho)$-periodic solutions $\hat x_1(t;\rho)$, $\hat x_2(t;\rho)$ of system \eqref{limsys}, where
\begin{gather*}
T(\rho)\equiv \int\limits_{x_-(\rho)}^{x_+(\rho)} \frac{\sqrt 2 \, d\varsigma}{\sqrt{\rho^{2h+2}-U(\varsigma)}},
\end{gather*}
and $x_-(\rho)<0<x_+(\rho)$ are roots to the equation $U(x)=\rho^{2h+2}$ as $\rho>\rho_0$ such that $U'(x_{\pm}(\rho))\neq 0$. For definiteness, suppose $\hat x_1(0;\rho) = x_+(\rho)$ and $\hat x_2(0;\rho) = 0$. It follows from~\cite[Lemma~4.1]{OS23DCDSB} that
\begin{gather}\label{nuas}
\nu(\rho)\equiv \frac{2\pi}{T(\rho)}, \quad \tilde \nu(\rho)\equiv \nu(\rho) \rho^{-h}\sim \sum_{k=0}^\infty \rho^{-k} \nu_k, \quad \rho\to\infty 
\end{gather}
with constant coefficients $\nu_k={\hbox{\rm const}}$. In particular,  
\begin{gather*}
 \nu_0=\frac{2\pi}{T_0}, \quad 
T_0=(2h+2)^{\frac{1}{2h+2}}\int\limits_{-1}^{1} \frac{\sqrt 2  d\zeta}{\sqrt{1-\zeta^{2h+2}}}, \quad \nu_1=0.
\end{gather*}

Define auxiliary $2\pi$-periodic functions $X_1(\varphi,\rho)=\hat x_1(\varphi/\nu(\rho),\rho)$ and $X_2(\varphi,\rho)=\hat x_2(\varphi/\nu(\rho),\rho)$. Then it follows from~\cite[Lemma~4.2]{OS23DCDSB} that
\begin{gather}\label{X1X2as}
\begin{split}
X_1(\varphi,\rho)& \sim \rho \sum_{k=0}^\infty \rho^{-k} X_{1,k}(\varphi), \\
X_2(\varphi,\rho)&\sim\rho^{h+1} \sum_{k=0}^\infty \rho^{-k} X_{2,k}(\varphi)
\end{split}
\end{gather}
as $\rho\to\infty$ with $2\pi$-periodic coefficients $X_{i,k}(\varphi)$, where the leading terms $X_{1,0}(\varphi)$ and $X_{2,0}(\varphi)$ satisfy the system 
\begin{gather}\label{Xi0}
\nu_0\partial_\varphi X_{1,0}=X_{2,0}, \quad \nu_0\partial_\varphi X_{2,0}=-X_{1,0}^{2h+1}, \quad \frac{X_{1,0}^{2h+2}}{2h+2}+\frac{X_{2,0}^2}{2}=1.
\end{gather}
Note that the series in \eqref{nuas} and \eqref{X1X2as} are assumed to be asymptotic as $\rho \to \infty$ (see, for example,~\cite[\S 1]{MVF89}). We use the functions  $X_1(\varphi,\rho)$, $X_2(\varphi,\rho)$ for rewriting system \eqref{ps} in the amplitude-angle variables $(\rho,\varphi)$: 
\begin{gather}\label{EAtr} 
	x_1=X_1(\varphi,\rho), \quad x_2=X_2(\varphi,\rho). 
\end{gather}
From the identity $H(X_1(\varphi,\rho),X_2(\varphi,\rho))\equiv \rho^{2h+2}$ it follows that 
\begin{gather*}
 \Delta(\rho):=   \begin{vmatrix}
        \partial_\varphi X_1 & \partial_\rho X_1\\
        \partial_\varphi X_2& \partial_\rho X_2
    \end{vmatrix} = \frac{(2h+2)\rho^{2h+1}}{\nu(\rho)}\neq 0, \quad \rho\geq \rho_0.
\end{gather*}
 Hence, the transformation \eqref{EAtr} is invertible for all $\rho\geq \rho_0$ and $\varphi\in[0,2\pi)$. Denote by
\begin{gather}\label{RPhi}
\rho=Y_1(x_1,x_2), \quad \varphi=Y_2(x_1,x_2)
\end{gather}
 the inverse transformation to \eqref{EAtr} and 
define the domain $\mathcal D(\rho_0) = \{(x_1,x_2)\in \mathbb R^2 : H(x_1,x_2) >\rho_0\}$. Then, we have the following:
\begin{Lem}\label{Lem1}
Let assumptions \eqref{Ucond} and \eqref{Qform} hold. Then, for all $(x_1,x_2)\in\mathcal D(\rho_0)$ system \eqref{ps} can be transformed into
\begin{gather}\label{PS}
\begin{split}
& d \rho = F_1(\rho,\varphi,S(t),t)\,dt + t^{-\gamma} \mu c_1(\rho,\varphi,S(t),t)\,dw(t),\\
& d \varphi = \left(\nu(\rho)+F_2(\rho,\varphi,S(t),t)\right) dt + t^{-\gamma} \mu c_2(\rho,\varphi,S(t),t)\,dw(t),
\end{split}
\end{gather}
with $F_i(\rho,\varphi,S,t) \equiv t^{-\alpha} f_i(\rho,\varphi,S)+t^{-2\gamma}  g_i(\rho,\varphi,S)$, $i\in\{1,2\}$, where
\begin{align*}
& f_i(\rho,\varphi,S) \equiv Q(X_1(\varphi,\rho),X_2(\varphi,\rho),S)\partial_{x_2} Y_i(X_1(\varphi,\rho),X_2(\varphi,\rho)), \\ 
& g_i(\rho,\varphi,S) \equiv \frac{\mu^2}{2}\sigma^2(X_1(\varphi,\rho),X_2(\varphi,\rho),S)\partial^2_{x_2}Y_i(X_1(\varphi,\rho),X_2(\varphi,\rho)), \\
& c_i(\rho,\varphi,S) \equiv \sigma(X_1(\varphi,\rho),X_2(\varphi,\rho),S)\partial_{x_2} Y_i(X_1(\varphi,\rho),X_2(\varphi,\rho)).
\end{align*}
Moreover, for every $i\in\{1,2\}$ the functions 
\begin{align*}
&\tilde f_i(\rho,\varphi,S)\equiv f_i(\rho,\varphi,S)\rho^{-(a+2-i)}, \quad\tilde g_i(\rho,\varphi,S)\equiv g_i(\rho,\varphi,S)\rho^{-(2b+2-i)},\\ 
&\tilde c_i(\rho,\varphi,S)\equiv c_i(\rho,\varphi,S)\rho^{-(b+2-i)}
\end{align*} 
have following asymptotic expansions:
\begin{gather}\label{fgAB}
\begin{split}
&\displaystyle \tilde f_i(\rho,\varphi,S)  \sim  \sum_{k=0}^\infty \rho^{-k} f_{i,k}(\varphi,S),  \quad
\displaystyle \tilde g_i(\rho,\varphi,S) \sim   \sum_{k=0}^\infty \rho^{-k} g_{i,k}(\varphi,S), \\
&\displaystyle \tilde c_i(\rho,\varphi,S) \sim  \sum_{k=0}^\infty \rho^{-k} c_{i,k}(\varphi,S)
\end{split}
\end{gather}
as $\rho\to\infty$ with $2\pi$-periodic coefficients $f_{i,k}(\varphi,S)$, $g_{i,k}(\varphi,S)$ and $c_{i,k}(\varphi,S)$. In particular,
\begin{align*}
&\begin{pmatrix} f_{1,0}(\varphi,S) \\ f_{2,0}(\varphi,S)\end{pmatrix}\equiv \frac{\nu_0  Q_{p,l}(S)  }{2(h+1)}(X_{1,0}(\varphi))^p (X_{2,0}(\varphi))^l
\begin{pmatrix} 
\partial_\varphi X_{1,0}(\varphi) \\ 
-X_{1,0}(\varphi)
\end{pmatrix},\\
&\begin{pmatrix} g_{1,0}(\varphi,S) \\ g_{2,0}(\varphi,S)\end{pmatrix}\equiv -\frac{ (\mu \nu_0 \sigma_{n,m}(S))^2  }{8(h+1)^2}(X_{1,0}(\varphi))^{2n}(X_{2,0}(\varphi))^{2m}
\begin{pmatrix} 
h(\partial_\varphi X_{1,0}(\varphi))^2+X_{1,0}(\varphi)\partial^2_\varphi X_{1,0}(\varphi)\\ 
-(h+2) X_{1,0}(\varphi)\partial_\varphi X_{1,0}(\varphi)
\end{pmatrix},\\
&\begin{pmatrix} c_{1,0}(\varphi,S) \\ c_{2,0}(\varphi,S)\end{pmatrix}\equiv \frac{\nu_0 \sigma_{n,m}(S)  }{2(h+1)} (X_{1,0}(\varphi))^{n}(X_{2,0}(\varphi))^{m}
\begin{pmatrix} 
\partial_\varphi X_{1,0}(\varphi) \\ 
-X_{1,0}(\varphi)
\end{pmatrix}.
\end{align*}
\end{Lem}

Note that the limiting system corresponding to \eqref{PS} has the following form: $d\varrho/dt=0$, $d\phi/dt=\nu(\varrho)$. Consider the resonance condition
\begin{gather}\label{rscond}
\nu(\varrho)=\varkappa^{-1} S'(t),
\end{gather}
where $\varkappa\in\mathbb Z_+$ corresponds to the resonance order. 
From \eqref{nuas} it follows that there exists $\tilde \rho_0>\rho_0$ such that $\nu(\varrho)>0$ and $\nu'(\varrho)>0$ for all $\varrho\geq \tilde \rho_0$. Hence, for all $\varkappa\in\mathbb Z_+$ there exists $\hat t_0>t_0$ such that equation \eqref{rscond} has a smooth solution $\varrho_\varkappa(t)\geq \tilde\rho_0$ defined for all $t\geq  \hat t_0$. Define $z(t)\equiv \varrho_\varkappa(t) t^{-{\beta}/{h}}$, then the following asymptotic expansion holds:
\begin{gather*}
 z(t)\sim\sum_{k=0}^\infty z_k t^{-\frac{k \beta}{h}}, \quad t\to\infty, \quad z_k={\hbox{\rm const}}, \quad z_0=\left(\frac{s(\beta+1)}{\nu_0 \varkappa}\right)^{\frac{1}{h}}, \quad z_1=-\frac{\nu_1}{h \nu_0}.
\end{gather*}
Hence, $\varrho_\varkappa(t)\sim t^{{\beta}/{h}} z_0$ as $t\to\infty$. Define the parameters
 \begin{gather}
\nonumber
	a=p+(l-1)(h+1), \quad  b=n+(m-1)(h+1),\\
\nonumber
	\mathcal M_1=-\alpha+\frac{a \beta}{h}  , \quad 
	\mathcal M_2=-2\gamma+\frac{2 b \beta}{h}, \quad 
	\mathcal M=\max\left\{\mathcal M_1,\mathcal M_2\right\}, \\
 \label{ABC} 
	A=\frac{1}{2}\left(\beta-\mathcal M\right), \quad 
	B=\beta+1, \quad 
	C=\frac{1}{2}\left(\beta-\min\left\{\mathcal M_1,\mathcal M_2\right\}\right),
\end{gather}
and consider the following axillary system, which is obtained from \eqref{PS} by dropping the stochastic parts:
\begin{gather}\label{trsys}
\frac{d\varrho}{dt}=F_1(\varrho,\phi,S(t),t), \quad \frac{d\phi}{dt}=\nu(\varrho)+F_2(\varrho,\phi,S(t),t).
\end{gather}

The following lemma describes the necessary condition for the existence of resonant refimes in the truncated system:
\begin{Lem}\label{Lem2}
Let system \eqref{trsys} has solutions with asymptotics $\varrho(t)\sim\varrho_\varkappa(t)$, $\phi(t)\sim \varkappa^{-1}S(t)$ as $t\to\infty$, then 
\begin{gather}\label{ncond}
	-1\leq \mathcal M<\beta.
\end{gather}
\end{Lem}
Thus, if the condition \eqref{ncond} is violated, the resonant solutions with unlimitedly growing amplitude and phase, synchronized with pumping, do not arise. We further assume that this condition is satisfied, and study the stability of such solutions with respect to white noise perturbations.

\section{Main results}
\label{MR}
Consider the additional assumption on the parameters
\begin{gather}\label{adas}
3\mathcal M_1-2\mathcal M_2\geq \beta.
\end{gather}
Combining this with \eqref{ncond}, we see that $\mathcal M_1>\mathcal M_2$, $\mathcal M=\mathcal M_1$, $ A=(\beta-\mathcal M_1)/2$, $C=(\beta-\mathcal M_2)/2$, $B\geq 2A>0$ and $C\geq 3A/2$. Let us set 
\begin{gather*}
\langle f(\zeta)\rangle_{\varkappa \zeta}:=\frac{1}{2\pi\varkappa} \int\limits_0^{2\pi\varkappa}f(\zeta)\,d\zeta, \quad
\{ f(\zeta)\}_{\varkappa \zeta}:= f(\zeta)-\langle f(\zeta)\rangle_{\varkappa \zeta}.
\end{gather*}
Then, we have the following:
\begin{Th}\label{Th1}
Let assumptions \eqref{Ucond}, \eqref{Qform}, \eqref{ncond}, and \eqref{adas} hold. Then, for all $\varkappa\in\mathbb Z_+$ there exist $\tilde t_0>t_0$, $\tilde \rho_0\geq \rho_0$, $R_0>0$ and the chain of transformations $(x_1,x_2,t)\mapsto (\rho,\varphi,t)\mapsto (r,\theta,\tau)\mapsto (R,\Theta,\tau)$, 
\begin{gather}
\nonumber x_1(t)=X_1(\varphi(t),\rho(t)), \quad x_2(t)=X_2(\varphi(t),\rho(t)),\\
\label{subs}
\rho(t)=\varrho_\varkappa(t)\left(1+t^{-A} r(\tau)\right), \quad \varphi(t)=\varkappa^{-1}S(t)+\theta(\tau), \quad \tau=\frac{t^{B}}{B}, \\
\nonumber r(\tau)=R(\tau)+\tau^{-\frac{A}{B}}\tilde Z_1(R(\tau),\Theta(\tau),\zeta(\tau),\tau), \quad 
\theta(\tau)=\Theta(\tau)+\tau^{-\frac{A}{B}}\tilde Z_2(R(\tau),\Theta(\tau),\zeta(\tau),\tau),
\end{gather}
with some smooth functions $\tilde Z_i(R,\Theta,\zeta,\tau)$ such that for all $t\geq \tilde t_0$ and $(x_1,x_2)\in \mathcal D(\tilde \rho_0)$ system \eqref{ps} can be transformed into
\begin{gather}\label{veq}
\begin{split}
&dR= \Lambda_1(R,\Theta,\zeta(\tau),\tau)\, d\tau+ \tau^{-\frac{C-A}{B}}\mu \eta_1(R,\Theta,\zeta(\tau),\tau)\,  d\tilde w(\tau),\\
&d\Theta= \Lambda_2(R,\Theta,\zeta(\tau),\tau)\, d\tau+ \tau^{-\frac{C}{B}}\mu \eta_2(R,\Theta,\zeta(\tau),\tau)\,  d\tilde w(\tau),
\end{split}
\end{gather}
where  $\tilde w(\tau)$ is a Wiener process on a probabiluty space $(\Omega,\mathfrak A,\mathbb P)$, $\zeta(\tau)\equiv S ((B\tau)^{\frac{1}{B}})$, 
\begin{gather}\label{Lambda12}\begin{split}
 \Lambda_1(R,\Theta,\zeta,\tau) \equiv & \sum_{K\in\{A,B-A,2A,B,2C-A\}}\tau^{-\frac{K}{B}}\Lambda_{1,K}(R,\Theta,\tau)+\tilde\Lambda_{1}(R,\Theta,\zeta,\tau), \\
\Lambda_2(R,\Theta,\zeta,\tau) \equiv & \sum_{K\in\{A,2A\}}\tau^{-\frac{K}{B}}\Lambda_{2,K}(R,\Theta,\tau)+\tilde\Lambda_{2}(R,\Theta,\zeta,\tau).
\end{split}
\end{gather}
The functions $\Lambda_{i}(R,\Theta,\zeta,\tau)$, $\eta_i(R,\Theta,\zeta,\tau)$, $\tilde Z_i(R,\Theta,\zeta,\tau)$ are $2\pi$-periodic in $\Theta$ and $2\pi\varkappa$-periodic in $\zeta$. Moreover, the following estimates hold: 
\begin{gather*}
\Lambda_{i,K}(R,\Theta,\tau)=  \langle \mathcal F_{i,K}(R,\Theta,\zeta,\tau)\rangle_{\varkappa \zeta}  +\delta_{i,1}\delta_{K,2A} \mathcal O(\tau^{-\frac{A}{B}})+\delta_{i,1}\delta_{K,2C-A} \mathcal O(\tau^{-\frac{2A}{B}}),\\
\tilde Z_{i}(R,\Theta,\zeta,\tau)=\mathcal O(1), \quad 
\eta_{i}(R,\Theta,\zeta,\tau)=\mathcal O(1), \quad 
\tilde\Lambda_{i}(R,\Theta,\zeta,\tau)=\mathcal O(\tau^{-\frac{3A}{B}})
\end{gather*} 
as $\tau\to\infty$ uniformly for all $|R|\leq R_0$ and $(\Theta,\zeta)\in\mathbb R^2$, where the functions $\mathcal F_{i,K}(r,\theta,\zeta,\tau)$ are defined by \eqref{Fforms}.
\end{Th}

Here, $\delta_{i,j}$ is the Kronecker delta. Note that \eqref{adas} is equivalent to the inequality $C\geq 3A/2$. If the opposite inequality holds, then the diffusion coefficients in \eqref{veq} become leading in the long-term dynamics as $\tau\to\infty$. This case requires a special attention and is not discussed here.  

Consider the truncated deterministic system corresponding to \eqref{veq} 
\begin{gather}\label{det1}
\frac{dR}{d\tau}=\Lambda_1(R,\Theta,\zeta(\tau),\tau), \quad 
\frac{d\Theta}{d\tau}=\Lambda_2(R,\Theta,\zeta(\tau),\tau)
\end{gather}
as $\tau\geq \tilde\tau_0$, where $\tilde \tau_0=\tilde t_0^B/B$.
It follows from \eqref{Lambda12} and \eqref{Fforms} that the functions $\Lambda_1(R,\Theta,\zeta,\tau)$ and $\Lambda_2(R,\Theta,\zeta,\tau)$ can be written as 
\begin{gather*}
\Lambda_i(R,\Theta,\zeta,\tau)\equiv \hat \lambda_i(R,\Theta,\tau)+\tilde \lambda_i(R,\Theta,\zeta,\tau), \quad i\in\{1,2\},
\end{gather*} 
with
\begin{gather}\label{lambda12}
\begin{split}
 \hat \lambda_1(R,\Theta,\tau)\equiv &\,\tau^{-\frac{A}{B}}\vartheta_{1,A}(\Theta,\tau)+\tau^{-\frac{B-A}{B}}\chi_{1,B-A}(\tau) \\
 &+ \tau^{-\frac{2A}{B}}\vartheta_{1,2A}(\Theta,\tau) R+\tau^{-1}\chi_{1,B}(\tau) R+ \tau^{-\frac{2C-A}{B}}\vartheta_{1,2C-A}(\Theta,\tau), \\
 \hat \lambda_2(R,\Theta,\tau)\equiv &\,\tau^{-\frac{A}{B}}\chi_{2,A} (\tau)R+ \tau^{-\frac{2A}{B}}\left(\vartheta_{2,2A}(\Theta,\tau) + \chi_{2,2A}(\tau) R^2\right),
\end{split}
\end{gather}
where
\begin{align*}
 \vartheta_{1,A}(\Theta,\tau)\equiv & \,  B^{-\frac{A}{B}}(z(t))^a   \left \langle  \tilde f_1\left(\varrho_\varkappa(t),\varkappa^{-1}\zeta+\Theta,\zeta\right)\right \rangle_{\varkappa \zeta}\Big|_{t=(B\tau)^{\frac 1B}}  , \\
\vartheta_{1,2A}(\Theta,\tau)\equiv & \, (a+1) B^{-\frac{A}{B}}\vartheta_{1,A}(\Theta)+B^{-\frac{2A}{B}}(z(t))^a\varrho_\varkappa(t)\left \langle   \partial_\rho \tilde f_1\left(\varrho_\varkappa(t),\varkappa^{-1}\zeta+\Theta,\zeta\right)\right \rangle_{\varkappa \zeta}\Big|_{t=(B\tau)^{\frac 1B}},\\ 
\vartheta_{1,2C-A}(\Theta,\tau)\equiv & \,  B^{\frac{A-2C}{B}}(z(t))^{2b}   \left \langle  \tilde g_1\left(\varrho_\varkappa(t),\varkappa^{-1}\zeta+\Theta,\zeta\right)\right \rangle_{\varkappa \zeta}\Big|_{t=(B\tau)^{\frac 1B}}  , \\ 
\vartheta_{2,2A} (\Theta,\tau)\equiv & \, B^{-\frac{2A}{B}} (z(t))^a\left\langle  \tilde f_2\left(\varrho_\varkappa(t),\varkappa^{-1}\zeta+\Theta,\zeta\right)\right\rangle_{\varkappa \zeta} \Big|_{t=(B\tau)^{\frac 1B}},\\ 
 \chi_{1,B-A}(\tau) \equiv & \, - B^{\frac{A-B}{B}}\left(\frac{\beta}{h}+ t (\log z(t))'\right)\Big|_{t=(B\tau)^{\frac 1B}} 
, \\
  \chi_{1,B}(\tau) \equiv & \,  \frac{A}{B} + B^{-\frac{A}{B}} \chi_{1,B-A}(\tau),\\
  \chi_{2,A}(\tau)\equiv & \, B^{-\frac{A}{B}}(z(t))^h\Big\{ h   \tilde \nu(\varrho_\varkappa(t)) + \partial_\rho\tilde \nu\left(\varrho_\varkappa(t)\right) \varrho_\varkappa(t) \Big\}\Big|_{t=(B\tau)^{\frac 1B}}, \\ 
  \chi_{2,2A} (\tau) \equiv & \, B^{-\frac{2A}{B}}  (z(t))^h \left\{\frac{h(h-1)   \tilde \nu\left(\varrho_\varkappa(t)\right)}{2} + h \partial_\rho\tilde \nu\left(\varrho_\varkappa(t) \right)\varrho_\varkappa(t)+ \frac{\partial_\rho^2\tilde \nu\left(\varrho_\varkappa(t)\right)(\varrho_\varkappa(t))^2}{2} \right\}\Big|_{t=(B\tau)^{\frac 1B}}.
\end{align*}
In this case, $\tilde \lambda_i(R,\Theta,\zeta,\tau)\equiv \Lambda_i(R,\Theta,\zeta,\tau)-\hat\lambda_i(R,\Theta,\tau)$. For every $i\in\{1,2\}$ the following estimate holds: $\tilde \lambda_i(R,\Theta,\zeta,\tau)=\mathcal O(\tau^{-{3A}/B})$ as $\tau\to\infty$ uniformly for all $|R|\leq R_0$ and $(\Theta,\zeta)\in\mathbb R^2$. It can easily be checked that the following asymptotic expansions hold:
\begin{gather}\label{asthetachi}\begin{split}
\vartheta_{i,K}(\Theta,\tau) & \sim \vartheta_{i,K,0}(\Theta)+\sum_{k=1}^\infty \tau^{-\frac{k \beta}{B h}} \vartheta_{i,K,k}(\Theta), \\ 
\chi_{i,K}(\tau) & \sim \chi_{i,K,0}+\sum_{k=1}^\infty \tau^{-\frac{k \beta}{B h}} \chi_{i,K,k}
\end{split}
\end{gather}
as $\tau \to\infty$, where $\vartheta_{i,K,k}(\Theta)$ are $2\pi$-periodic in $\Theta$ and $\chi_{i,K,k}$ are constants. In particular, the leading terms have the form
\begin{align*}
	\vartheta_{1,A,0}(\Theta)\equiv &\,   B^{-\frac{A}{B}} z_0^a  \left \langle   f_{1,0}\left(\varkappa^{-1}\zeta+\Theta,\zeta\right)\right \rangle_{\varkappa \zeta}, \\
	\vartheta_{1,2A,0}(\Theta)\equiv &\, (a+1) B^{-\frac{A}{B}}\vartheta_{1,A,0}(\Theta),\\ 
	\vartheta_{1,2C-A,0} (\Theta)\equiv &\,  B^{\frac{A-2C}{B}} z_0^{2b}  \left\langle  g_{1,0}\left(\varkappa^{-1}\zeta+\Theta,\zeta\right)\right\rangle_{\varkappa \zeta}, \\
	\vartheta_{2,2A,0} (\Theta)\equiv &\, B^{-\frac{2A}{B}} z_0^a  \left\langle  f_{2,0}\left(\varkappa^{-1}\zeta+\Theta,\zeta\right)\right\rangle_{\varkappa \zeta},\\ 
\chi_{1,B-A,0} = &\, - B^{\frac{A-B }{B}} \frac{\beta}{h}, \\ 
\chi_{1,B,0} = &\, \frac{1}{B} \left( A- \frac{\beta}{  h}\right), \\
\chi_{2,A,0}=  &\, B^{-\frac{A}{B}}  z_0^h h  \nu_0, \\ 
\chi_{2,2A,0} = &\, B^{-\frac{2A}{B}} \frac{z_0^h h(h-1)\nu_0}{2}.
\end{align*}

It is clear that the long-term dynamics of the asymptotically autonomous system \eqref{det1} depends on the properties of the leading terms of the right-hand sides as $\tau \to \infty$. With this in mind, define 
\begin{align}\label{Pdef}
 \mathcal P(\Theta)&\equiv  \vartheta_{1,A,0}(\Theta)+\delta_{B,2A}\chi_{1,B-A,0}, \\
\label{Jdef}
 J(\Theta)&\equiv  \vartheta_{1,2A,0}(\Theta)+\vartheta_{2,2A,0}'(\Theta)+\delta_{B,2A}\chi_{1,B,0},
\end{align}  
and consider the assumption 
\begin{gather}\label{astheta}
\exists\, \Theta_0\in\mathbb R: \quad \mathcal P(\Theta_0)=0, \quad \mathcal P'(\Theta_0)\neq 0.
\end{gather}
Then, we have the following:
\begin{Lem}\label{LTh2}
Let assumptions \eqref{ncond}, \eqref{adas} and \eqref{astheta} hold. 
\begin{itemize}
\item
If $\mathcal P'(\Theta_0)<0$ and $J(\Theta_0)<0$, then there exists $\tau_0\geq \tilde\tau_0$ such that system \eqref{det1} has an asymptotically stable solution $R_\ast(\tau)$, $\Theta_\ast(\tau)$ defined for all $\tau\geq \tau_0$. Moreover, $R_\ast(\tau)=o(1)$ and $\Theta_\ast(\tau)= \Theta_0+o(1)$ as $\tau\to\infty$. 
\item If $\mathcal P'(\Theta_0)<0$ and $J(\Theta_0)>0$ or $\mathcal P'(\Theta_0)>0$, then there exists $\varepsilon_\ast>0$ such that for all $\delta\in (0,\varepsilon_\ast)$ the solutions of system \eqref{det1} with initial data $\sqrt{R^2(\tau_\ast)+ (\Theta(\tau_\ast)-\Theta_0)^2}=\delta$ and some $\tau_\ast\geq \tilde\tau_0$ exit from the domain $\{(R,\Theta)\in\mathbb R^2: \sqrt{R^2+(\Theta-\Theta_0)^2}< \varepsilon_\ast\}$ in a finite time.
\end{itemize}
\end{Lem}
Thus, Lemma~\ref{LTh2} guarantees the existence and stability of the phase-locking regime in the truncated system \eqref{det1}. Such solutions correspond to a resonant increase of the amplitude, $\rho(t)\sim \varrho_\varkappa(t)$, and to a synchronization of the system with the excitation, $\varphi(t)-\varkappa^{-1}S(t)\sim\Theta_0$ as $t\to\infty$. If \eqref{astheta} does not hold, we have the following:
\begin{Lem}\label{LTh3}
Let assumptions \eqref{ncond}, \eqref{adas} hold and 
\begin{gather}\label{asthetanot}
\mathcal P(\Theta)\neq 0 \quad \forall\, \Theta\in\mathbb R.
\end{gather}
Then, the solutions of system \eqref{det1} exit from any bounded domain in a finite time.
\end{Lem}
This lemma describes the conditions for the phase drifting mode. In this case, the amplitude $\rho(t)$ and the phase $\varphi(t)$ of system can significantly differ from $\varrho_\varkappa(t)$ and $\varkappa^{-1}S(t)$ even in the absence of the stochastic perturbations. 

Let us show that the phase-locking regime is preserved with a high probability in the full stochastic system \eqref{veq}. Define the function
\begin{gather*}
d(\rho,\varphi,t)\equiv \sqrt{\left\{\left(\frac{\rho}{\varrho_{\varkappa}(t)}-1\right)t^{A}-R_\ast\left(\frac{t^B}{B}\right)\right\}^2+\left\{\varphi-\varkappa^{-1}S(t)-\Theta_\ast\left(\frac{t^B}{B}\right)\right\}^2}.
\end{gather*}
Then, we have
\begin{Th}\label{Th2}
Let assumptions \eqref{Ucond}, \eqref{Qform}, \eqref{ncond}, \eqref{adas}, and \eqref{astheta} hold. If $\mathcal P'(\Theta_0)<0$ and $J(\Theta_0)<0$, then for all $\varepsilon_1>0$, $\varepsilon_2>0$, $\epsilon\in (0,1)$ and $t_\ast\geq \tilde t_0$ there exist $\delta_1>0$ and $\delta_2>0$ such that the solution $\rho(t)$, $\varphi(t)$ of system \eqref{PS} with initial data  $d(\rho(t_\ast),\varphi(t_\ast),t_\ast)<\delta_1$ and $0<\mu<\delta_2$ satisfies
\begin{gather}\label{Th2est}
\mathbb P\left(\sup_{0<t-t_\ast\leq \tilde{\mathcal T}} d(\rho(t),\varphi(t),t)\geq\varepsilon_1\right)\leq \varepsilon_2,
\end{gather}
where 
\begin{gather*}
\tilde{\mathcal T}=
\begin{cases}
\displaystyle t_\ast\left(\mu^{-\frac{2(1-\epsilon)}{B}}-1\right), & \text{if} \quad  C < A+\frac{B}{2},\\
\displaystyle  t_\ast \left(\exp{ \mu^{-\frac{2(1-\epsilon)}{B}}}-1\right), & \text{if} \quad C  = A+\frac{B}{2},\\
\infty, & \text{if} \quad C  > A+\frac{B}{2}.
\end{cases}
\end{gather*}
\end{Th}
Note that deviation estimates similar to \eqref{Th2est} usually arise when studying the stability in probability~\cite[\S 5.3]{RH12}. Besides, the stability on an asymptotically large time interval corresponds to some variant of a practical stability~\cite[\S 25]{LL61}. 

Thus, Theorem~\ref{Th2} describes the condtions under which the stochastic perturbations do not destroy the resonant capture in system \eqref{ps}. Depending on the parameters of the system, the dynamics $\rho(t)\approx \varrho_\varkappa(t)$, $\varphi(t)\approx \varkappa^{-1}S(t)+\Theta_0$ is preserved with a high probability on infinite or asymptotically large time intervals as $\mu\to 0$. Combining this with Theorem~\ref{Th1}, we see that the resonant behaviour occurs in system \eqref{PS} such that 
\begin{gather*}
x_1(t)\approx t^{\frac{\beta}{h}} z_0 X_{1,0}(\varphi(t)), \quad 
x_2(t)\approx t^{\frac{(h+1)\beta}{h}} z_0^{h+1} X_{2,0}(\varphi(t)),
\end{gather*}
where $X_{1,0}(\varphi)$ and $X_{2,0}(\varphi)$ are $2\pi$-periodic functions defined by \eqref{Xi0}.

\section{Proof of auxiliary results}
\label{Sec4}
\begin{proof}[Proof of Lemma~\ref{Lem1}]
Applying It\^{o}'s formula~\cite[\S 4.2]{BO98} to \eqref{RPhi}, we see that in the new variables $(\rho,\varphi)$ system \eqref{ps} takes the form \eqref{PS}.
From the properties of the functions $X_1(\varphi,\rho)$ and $X_2(\varphi,\rho)$ it follows that
\begin{align*}
	\partial_{x_2} \begin{pmatrix} Y_1 \\ Y_2 \end{pmatrix} & \equiv \frac{1}{ \Delta(\rho)} \begin{pmatrix} \partial_\varphi X_1\\ -\partial_\rho X_1 \end{pmatrix}, \\ 
	\partial_{x_2}^2 \begin{pmatrix} Y_1 \\ Y_2 \end{pmatrix} & \equiv \frac{1}{ \Delta(\rho)} \left( \partial_\varphi X_1 \partial_\rho -  \partial_\rho X_1  \partial_\varphi \right)  \frac{1}{ \Delta(\rho)} \begin{pmatrix} \partial_\varphi X_1\\ -\partial_\rho X_1 \end{pmatrix}.
\end{align*}
Combing this with \eqref{Qform} and \eqref{X1X2as}, we obtain \eqref{fgAB}.
\end{proof}

\begin{proof}[Proof of Lemma~\ref{Lem2}]
It follows from \eqref{fgAB} that there exist $D>0$ and $\hat \varrho_0>\hat\rho_0$ such that
\begin{gather*}
|f_i(\varrho,\phi,S)|\leq D \varrho^{a+2-i}, \quad |g_i(\varrho,\phi,S)|\leq D \varrho^{2b+2-i} 
\end{gather*}
for all $i\in\{1,2\}$, $\varrho\geq \hat \varrho_0$ and $(\varrho,\phi)\in\mathbb R^2$. Combining this with the first equation in \eqref{trsys}, we obtain the following inequality for solutions:
\begin{gather}\label{rhoineq}
	\left|\frac{d}{dt}\log\varrho(t)\right| \leq D \left(t^{-\alpha} (\varrho(t))^{a}+t^{-2\gamma} (\varrho(t))^{2b}\right).
\end{gather}
It is not hard to check that 
\begin{gather*}
\left(\log\varrho_\varkappa(t)\right)'\sim \frac{\beta}{h} t^{-1}, \quad 
t^{-\alpha} (\varrho_\varkappa(t))^{a}\sim z_0^a t^{\mathcal M_1}, \quad 
t^{-2\gamma} (\varrho_\varkappa(t))^{2b}\sim z_0^{2b} t^{\mathcal M_2}, \quad t\to\infty.
\end{gather*}
Substituting these estimates into \eqref{rhoineq}, we see that system \eqref{trsys} admits solutions with $\varrho(t)\sim \varrho_\varkappa(t)$ as $t\to\infty$ if $\mathcal M\geq -1$.

From the second equation in \eqref{trsys} it follows that
\begin{gather*}
\frac{\phi(t)}{S(t)}-\varkappa^{-1}= \mathcal O\left(t^{\mathcal M-\beta }\right), \quad t\to\infty
\end{gather*}
for $\varrho(t)\sim \varrho_\varkappa(t)$. Since  $\phi(t)/S(t)\to \varkappa^{-1}$ as $t\to\infty$, we get $\mathcal M<\beta $.
\end{proof}

\section{Change of variables}
\label{Sec5}
\subsection{Amplitude remainder and phase shift}
Consider the change of variables \eqref{subs} in system \eqref{PS}. Using It\^{o}'s formula and the change of time formula for stochastic integrals~\cite[\S 8.5]{BO98}, we obtain
\begin{gather}\label{zeq}
\begin{split}
&dr=\mathcal F_1(r,\theta,\zeta(\tau),\tau)\, d\tau + \tau^{-\frac{C-A}{B}}\mu \xi_1(r,\theta,\zeta(\tau),\tau)\, d\tilde w(\tau), \\
&d\theta=\mathcal F_2(r,\theta,\zeta(\tau),\tau)\, d\tau + \tau^{-\frac{C}{B}}\mu \xi_2(r,\theta,\zeta(\tau),\tau)\, d\tilde w(\tau),
\end{split}
\end{gather}
where the functions  
\begin{align*}
\mathcal F_1 \equiv &\,    t^{A+1-B}  \left\{F_1\left(\varrho_\varkappa(t) (1+t^{-A} r ),\varkappa^{-1}\zeta+\theta,\zeta,t\right) -\varrho_\varkappa'(t)\right\}(\varrho_\varkappa(t))^{-1}\\ & + t^{-B} r \left(A-t\left(\log\varrho_\varkappa(t)\right)'\right)\Big|_{t=(B\tau)^{\frac 1B}},\\
\mathcal F_2 \equiv &\,  t^{1-B} \Big\{\nu\big(\varrho_\varkappa(t) (1+t^{-A} r )\big)-\varkappa^{-1}S'(t)+F_2\left(\varrho_\varkappa(t) (1+t^{-A} r ),\varkappa^{-1}\zeta+\theta,\zeta,t\right)\Big\}\Big|_{t=(B\tau)^{\frac 1B}},\\
\xi_1\equiv&\, t^{C-\gamma+\frac{1-B}{2}}B^{-\frac{C-A}{B}}c_1\left(\varrho_\varkappa(t) (1+t^{-A} r ),\varkappa^{-1}\zeta+\theta,\zeta,t\right)(\varrho_\varkappa(t))^{-1}\Big|_{t=(B\tau)^{\frac 1B}}, \\
\xi_2\equiv &\, t^{C-\gamma+\frac{1-B}{2}} B^{-\frac CB} c_2\left(\varrho_\varkappa(t) (1+t^{-A} r ),\varkappa^{-1}\zeta+\theta,\zeta,t\right)\Big|_{t=(B\tau)^{\frac 1B}}
\end{align*}
are $2\pi$-periodic in $\theta$ and $2\pi\varkappa$-periodic in $\zeta$.
Note that the right-hand sides of system \eqref{zeq} satisfy the following estimates:
\begin{gather}\label{FFxixias}
	\begin{split}
\mathcal F_1(r,\theta,\zeta,\tau)  = & \sum_{K\in\{A,B-A,2A,2C-A,B\}}
  \tau^{-\frac{K}{B}}\mathcal F_{1,K}(r,\theta,\zeta,\tau) +\mathcal O(\tau^{-\frac{3A}{B}}), \\
\mathcal F_2(r,\theta,\zeta,\tau)  = & \sum_{K\in\{A,2A\}}
  \tau^{-\frac{K}{B}}\mathcal F_{2,K}(r,\theta,\zeta,\tau) +\mathcal O(\tau^{-\frac{3A}{B}}), \\
\xi_i(r,\theta,\zeta,\tau)  = & \,  \xi_{i,0}(r,\theta,\zeta,\tau)+\mathcal O(\tau^{-\frac{A}{B}})
\end{split}
\end{gather}
as $\tau\to\infty$ uniformly for all $(\theta,\zeta)\in\mathbb  R^2$ and $|r|\leq r_0$, where $r_0=(\tilde \rho_0/\rho_0-1)\hat t_0^A>0$, 
\begin{gather}\label{Fforms}
\begin{split}
\mathcal F_{1,A}(r,\theta,\zeta,\tau)\equiv &\,  B^{-\frac{A}{B}}(z(t))^a \tilde f_1\left(\varrho_\varkappa(t) (1+t^{-A} r ),\varkappa^{-1}\zeta+\theta,\zeta\right)\Big|_{t=(B\tau)^{\frac 1B}},\\
\mathcal F_{1,B-A}(r,\theta,\zeta,\tau)\equiv &\, - B^{\frac{A-B}{B}}\left(\frac{\beta}{h}+ t (\log z(t))'\right)\Big|_{t=(B\tau)^{\frac 1B}},\\
\mathcal F_{1,2A}(r,\theta,\zeta,\tau)\equiv &\, r B^{-\frac{A}{B}} (a+1)  \mathcal F_{1,A}(r,\theta,\zeta,\tau),\\
\mathcal F_{1,B}(r,\theta,\zeta,\tau)\equiv &\, r\left(\frac{A}{B} + B^{-\frac{A}{B}} \mathcal F_{1,B-A}(r,\theta,\zeta,\tau)\right),\\
\mathcal F_{1,2C-A}(r,\theta,\zeta,\tau)\equiv &\,  B^{\frac{A-2C}{B}}(z(t))^{2b} \tilde g_1\left(\varrho_\varkappa(t) (1+t^{-A} r ),\varkappa^{-1}\zeta+\theta,\zeta\right)\Big|_{t=(B\tau)^{\frac 1B}},\\
\mathcal F_{2,A}(r,\theta,\zeta,\tau)\equiv &\,  B^{-\frac{A}{B}}(z(t))^h\Big\{ h r \tilde \nu(\varrho_\varkappa(t)) + t^{A}\Big(\tilde \nu\big(\varrho_\varkappa(t)(1+t^{-A}r)\big)-\tilde \nu\big(\varrho_\varkappa(t)\big) \Big)\Big\}\Big|_{t=(B\tau)^{\frac 1B}}, \\
\mathcal F_{2,2A}(r,\theta,\zeta,\tau)\equiv &\, r  B^{-\frac{2A}{B}} (z(t))^h\frac{h}{2} \Big\{ (h-1) r \tilde \nu(\varrho_\varkappa(t)) + 2 t^{A}\Big(\tilde \nu\big(\varrho_\varkappa(t)(1+t^{-A}r)\big)-\tilde \nu\big(\varrho_\varkappa(t)\big) \Big)\Big\} \\
&\, +  B^{-\frac{2A}{B}} (z(t))^a\tilde f_2\left(\varrho_\varkappa(t) (1+t^{-A} r ),\varkappa^{-1}\zeta+\theta,\zeta\right)\Big|_{t=(B\tau)^{\frac 1B}},\\
\xi_{1,0}(r,\theta,\zeta,\tau)\equiv &\, B^{-\frac{C-A}{B}} (z(t))^b \tilde c_1\left(\varrho_\varkappa(t) (1+t^{-A} r ),\varkappa^{-1}\zeta+\theta,\zeta\right)  \Big|_{t=(B\tau)^{\frac 1B}},\\
\xi_{2,0}(r,\theta,\zeta,\tau)\equiv &\, B^{-\frac{C}{B}} (z(t))^b \tilde c_2\left(\varrho_\varkappa(t) (1+t^{-A} r ),\varkappa^{-1}\zeta+\theta,\zeta\right)  \Big|_{t=(B\tau)^{\frac 1B}}.
\end{split}
\end{gather}
Moreover, the following asymptotic expansions hold:
\begin{gather*}
\mathcal F_{i,K}(r,\theta,\zeta,\tau)\sim \mathcal F_{i,K}^0(r,\theta,\zeta)+\sum_{j\geq 1, k\geq 0} \tau^{-\frac{j (\beta/h)+k A}{B}}\mathcal F_{i,K}^{j,k}(r,\theta,\zeta),\\
\xi_{i,0}(r,\theta,\zeta,\tau)\sim \mathcal \xi_{i,0}^0(r,\theta,\zeta)+\sum_{j\geq 1, k\geq 0} \tau^{-\frac{j (\beta/h)+k A}{B}} \xi_{i,0}^{j,k}(r,\theta,\zeta)
\end{gather*}
as $\tau\to\infty$ with time-independent coefficients $\mathcal F_{1,K}^{j,k}(r,\theta,\zeta)$ and $\xi_{i,0}^{j,k}(r,\theta,\zeta)$. In particular,  
\begin{align*}
	\mathcal F_{1,A}^0 & \equiv  B^{-\frac{A}{B}} z_0^a f_{1,0}(\varkappa^{-1}\zeta+\theta,\zeta), & \mathcal F_{2,A}^0 & \equiv r B^{-\frac{A}{B}} h  \nu_0 z_0^h,\\ 
	\mathcal F_{1,B-A}^0 & \equiv -B^{\frac{A-B}{B}}\frac{\beta}{h},  & \mathcal F_{2,B-A}^0 & \equiv 0, \\ 
	\mathcal F_{1,2A}^0 & \equiv r B^{-\frac{A}{B}}(a+1)  \mathcal F_{1,A}^0, & \mathcal F_{2,2A}^0 & \equiv B^{-\frac{2A}{B}} \left(h(h-1)\nu_0 z_0^{h}\frac{r^2 }{2} + f_{2,0}(\varkappa^{-1}\zeta+\theta,\zeta) z_0^a \right), \\ 
	\mathcal F_{1,B}^0 & \equiv \frac{r}{B}\left(A-\frac{\beta}{h}\right),   & \mathcal F_{2,B}^0 & \equiv 0,\\
\mathcal F_{1,2C-A}^0 & \equiv  B^{\frac{A-2C}{B}}z_0^{2b}g_{1,0}(\varkappa^{-1}\zeta+\theta,\zeta),  & \mathcal F_{2,2C-A}^0 & \equiv 0, \\
	\xi_{1,0}^0  & \equiv B^{-\frac{C-A}{B}} z_0^b c_{1,0}(\varkappa^{-1}\zeta+\theta,\zeta), &
	\xi_{2,0}^0 & \equiv B^{-\frac{C}{B}} z_0^b c_{2,0}(\varkappa^{-1}\zeta+\theta,\zeta).
\end{align*}

\subsection{Averaging}
Note that the limiting system, corresponding to \eqref{zeq}, has the form
\begin{gather*}
\frac{dr}{d\tau}=0, \quad \frac{d\theta}{d\tau}=0, \quad \frac{d\zeta}{d\tau}=\varsigma,
\end{gather*}
where $\varsigma=s B>0$. In this case, $\zeta$ can be considered as a fast variable, and system \eqref{zeq} can be simplified by averaging the drift terms with respect to $\zeta$. Note that such averaging transformations are successfully used in problems with a small parameter~\cite{BM61,AKN06} and in studying the asymptotics of solutions to non-autonomous systems at infinity~\cite{LK09,OS23JMS}. Consider a near-identity transformation in the following form:
\begin{gather}\label{trans}\begin{split}
 &\tilde R(r,\theta,\zeta,\tau)
\equiv \displaystyle r +\sum_{K\in\mathcal X}\tau^{-\frac{K}{B}} Z_{1,K}(r,\theta,\zeta,\tau), \\
 &\tilde \Theta(r,\theta,\zeta,\tau)
\equiv \displaystyle \theta +\sum_{K\in\mathcal X}\tau^{-\frac{K}{B}} Z_{2,K}(r,\theta,\zeta,\tau),
\end{split}
\end{gather}
where $\mathcal X:=\{A,B-A,2A,2(B-A),2C-A,B,B+2C-3A\}$. The coefficients $Z_{i,K}(r,\theta,\zeta,\tau)$, $i\in\{1,2\}$ are assumed to be periodic w.r.t. $\theta$ and $\zeta$ with zero means. These functions are chosen in such a way that the drift terms of system \eqref{veq} written in the new variables 
\begin{gather}\label{RPsi}
R(\tau)=\tilde R(r(\tau),\theta(\tau),\zeta(\tau),\tau), \quad 
\Theta(\tau)=\tilde \Theta(r(\tau),\theta(\tau),\zeta(\tau),\tau)
\end{gather}
do not depend explicitly on $\zeta(\tau)$ at least in the leading terms.
By applying It\^{o}'s formula~\cite[\S 4.2]{BO98} to \eqref{RPsi}, we obtain  
\begin{align*}
& dR = \mathcal L \tilde R \, d\tau+  \mu\left(\tau^{-\frac{C-A}{B}} \xi_1 \partial_r  +\tau^{-\frac{C}{B}} \xi_2 \partial_\theta\right) \tilde R \, d\tilde w(\tau),\\
& d\Theta = \mathcal L \tilde \Theta \, d\tau+  \mu\left(\tau^{-\frac{C-A}{B}} \xi_1 \partial_r  +\tau^{-\frac{C}{B}} \xi_2 \partial_\theta\right) \tilde \Theta \, d\tilde w(\tau),
\end{align*}
where
\begin{gather*}
\mathcal L := \mathcal F_1 \partial_r +  \mathcal F_2 \partial_\theta + \partial_\tau+\varsigma\partial_\zeta+\frac{\mu^2}{2}\tau^{-\frac{2(C-A)}{B}}\left(\xi_1^2 \partial^2_r+2\tau^{-\frac{A}{B}}\xi_1\xi_2\partial_r\partial_\theta+\tau^{-\frac{2A}{B}}\xi_2^2 \partial_\theta^2\right)
\end{gather*}
is the generator of the process defined by \eqref{zeq} (see, for example,~\cite[\S 3.3]{RH12}). Calculating $\mathcal L \tilde R$ and $\mathcal L \tilde \Theta$, using \eqref{FFxixias} and \eqref{trans}, and comparing the result with \eqref{Lambda12}, we obtain the following chain of equations:
\begin{gather}\label{vsys}
\varsigma\partial_\zeta  Z_{i,K} =\Lambda_{i,K}(r,\theta,\tau)-\mathcal F_{i,K}(r,\theta,\zeta,\tau)+\tilde{\mathcal F}_{i,K}(r,\theta,\zeta,\tau), \quad K\in\mathcal X, \quad i\in \{1,2\}.
\end{gather}
It is assumed that 
\begin{gather}\label{aaa1}
\mathcal F_{2,B-A}\equiv \mathcal F_{2,B}\equiv \mathcal F_{2,2C-A}\equiv\mathcal F_{i,2(B-A)}\equiv\mathcal F_{i,B+2C-3A}\equiv 0.
\end{gather} 
The functions $\tilde{\mathcal F}_{i,K}(r,\theta,\zeta,\tau)$ are expressed through $\{Z_{i,k}, \Lambda_{i,k}\}_{k<K}$. 
In particular,  
\begin{gather}\label{aaa2}
\begin{split}
\tilde{\mathcal F}_{i,A}(r,\theta,\zeta,\tau)\equiv & \, 0, \\ 
\tilde {\mathcal F}_{i,B-A}(r,\theta,\zeta,\tau) \equiv & \, 0, \\
 \tilde{\mathcal F}_{i,2A}(r,\theta,\zeta,\tau)\equiv &  
\left(
Z_{1,A}(r,\theta,\zeta,\tau) \partial_r+  Z_{2,A}(r,\theta,\zeta,\tau)\partial_\theta 
\right)\Lambda_{i,A}(r,\theta,\tau) 
\\ & - \left(
\mathcal F_{1,A}(r,\theta,\zeta,\tau)\partial_r + \mathcal F_{2,A}(r,\theta,\zeta,\tau)\partial_\theta
\right) Z_{i,A}(r,\theta,\zeta,\tau), \\
 \tilde{\mathcal F}_{i,2(B-A)}(r,\theta,\zeta,\tau)\equiv &  
\left(
Z_{1,B-A}(r,\theta,\zeta,\tau) \partial_r+  Z_{2,B-A}(r,\theta,\zeta,\tau)\partial_\theta 
\right)\Lambda_{i,B-A}(r,\theta,\tau) 
\\ & -  
\mathcal F_{1,B-A}(r,\theta,\zeta,\tau)\partial_r   Z_{i,A}(r,\theta,\zeta,\tau),\\
\tilde{\mathcal F}_{i,2C-A}(r,\theta,\zeta,\tau)\equiv & 
- \frac{\mu^2}{2}(\xi_{1,0}(r,\theta,\zeta,\tau))^2\partial_r^2 Z_{i,A}(r,\theta,\zeta,\tau),\\
 \tilde{\mathcal F}_{i,B}(r,\theta,\zeta,\tau)\equiv &  
\left(
Z_{1,B-A}(r,\theta,\zeta,\tau) \partial_r+  Z_{2,B-A}(r,\theta,\zeta,\tau)\partial_\theta 
\right)\Lambda_{i,A}(r,\theta,\tau) 
\\ 
&+ \left(
Z_{1,A}(r,\theta,\zeta,\tau) \partial_r+  Z_{2,A}(r,\theta,\zeta,\tau)\partial_\theta 
\right)\Lambda_{i,B-A}(r,\theta,\tau) 
\\
& - \left(
\mathcal F_{1,A}(r,\theta,\zeta,\tau)\partial_r + \mathcal F_{2,A}(r,\theta,\zeta,\tau)\partial_\theta
\right) Z_{i,B-A}(r,\theta,\zeta,\tau) \\
& - 
\mathcal F_{1,B-A}(r,\theta,\zeta,\tau)\partial_r Z_{i,A}(r,\theta,\zeta,\tau),\\
\tilde{\mathcal F}_{i,B+2C-3A}(r,\theta,\zeta,\tau)\equiv & 
- \frac{\mu^2}{2}(\xi_{1,0}(r,\theta,\zeta,\tau))^2\partial_r^2 Z_{i,B-A}(r,\theta,\zeta,\tau).
\end{split}
\end{gather}
Define 
\begin{gather*}
\Lambda_{i,K}(r,\theta,\tau)\equiv \langle \mathcal F_{i,K}(r,\theta,\zeta,\tau)-\tilde{\mathcal F}_{i,K}(r,\theta,\zeta,\tau)\rangle_{\varkappa\zeta}, \quad i\in\{1,2\}, \quad K\in\mathcal X.
\end{gather*}
Then, the right-hand side of system \eqref{vsys} has zero average. Integrating \eqref{vsys}, we obtain
\begin{align*}
Z_{i,K}(r,\theta,\zeta,\tau)\equiv -\frac{1}{\varsigma}\left\{\int\limits_0^\zeta \{\mathcal F_{i,K}(r,\theta,\zeta,\tau)-\tilde{\mathcal F}_{i,K}(r,\theta,\zeta,\tau)\}_{\varkappa \zeta} \, d\zeta\right\}_{\varkappa \zeta}.
\end{align*}
It can easily be checked that 
\begin{align*}
&  Z_{2,A} \equiv 
Z_{i,B-A} \equiv 
Z_{2,2(B-A)} \equiv 
Z_{2,2C-A} \equiv 
Z_{2,B} \equiv 
Z_{i,B+2C-3A} \equiv0, \\
&\tilde {\mathcal F}_{2,2(B-A)} \equiv 
\tilde {\mathcal F}_{2,2C-A} \equiv 
\tilde {\mathcal F}_{2,B} \equiv
\tilde {\mathcal F}_{i,B+2C-3A} \equiv 0, \\
&\langle\tilde {\mathcal F}_{2,2A}(r,\theta,\zeta,\tau)\rangle_{\varkappa \zeta}\equiv 
\langle\tilde {\mathcal F}_{i,2(B-A)}(r,\theta,\zeta,\tau)\rangle_{\varkappa \zeta}\equiv 
\langle\tilde {\mathcal F}_{1,B}(r,\theta,\zeta,\tau)\rangle_{\varkappa \zeta}\equiv0,
\end{align*}
and $\langle\tilde {\mathcal F}_{1,2A}(r,\theta,\zeta,\tau)\rangle_{\varkappa \zeta}=\mathcal O(\tau^{-A/B})$, 
$\langle\tilde {\mathcal F}_{1,2C-A}(r,\theta,\zeta,\tau)\rangle_{\varkappa \zeta}=\mathcal O(\tau^{-2A/B})$  as  $\tau \to\infty$ uniformly for all $|r|\leq r_0$ and $\theta\in\mathbb R$.
Combining this with \eqref{aaa1} and \eqref{aaa2}, we obtain 
\begin{align*}
&\Lambda_{i,2(B-A)}(r,\theta,\tau)\equiv \Lambda_{i,B+2C-3A}(r,\theta,\tau)\equiv  0, \quad i\in\{1,2\},\\
& \Lambda_{2,B-A}(r,\theta,\tau)\equiv \Lambda_{2,B}(r,\theta,\tau)\equiv \Lambda_{2,2C-A}(r,\theta,\tau)\equiv 0, \\
& \Lambda_{i,K}(r,\theta,\tau)= \langle \mathcal F_{i,K}(r,\theta,\zeta,\tau)\rangle_{\varkappa \zeta}+\delta_{i,1}\delta_{K,2A}\mathcal O(\tau^{-\frac{A}{B}})+\delta_{i,1}\delta_{K,2C-A}\mathcal O(\tau^{-\frac{2A}{B}}),\quad \tau \to\infty.
\end{align*} 

It follows from \eqref{trans} that for all $\varepsilon \in (0,r_0)$ there exists  $\tilde\tau_0\geq \hat t_0^B/B$ such that
\begin{gather*}
|\tilde R(r, \theta, \zeta, \tau) - r|\leq \varepsilon,  \quad | \tilde \Theta(r, \theta,\zeta, \tau) - \theta|\leq \varepsilon
\end{gather*}
for all $|r|\leq r_0$, $(\theta,\zeta)\in\mathbb R^2$ and $\tau\geq \tilde \tau_0$. Moreover, 
\begin{gather*}
{\hbox{\rm det}}\frac{\partial (\tilde R, \tilde \Theta)}{\partial (r,\theta)}\equiv 
\begin{vmatrix}
\partial_r\tilde R & \partial_\theta \tilde R \\
\partial_r\tilde \Theta & \partial_\theta \tilde \Theta
\end{vmatrix} = 1+\mathcal O(\tau^{-\frac{A}{B}}), \quad \tau\to\infty
\end{gather*}
uniformly for all $|r|\leq r_0$ and $\theta\in\mathbb R$. Hence, the transformation $(r,\theta)\mapsto (R,\Theta)$ is invertible. Denote by $r = \tilde r(R,\Theta,\zeta,\tau)$, $\theta = \tilde \theta(R,\Theta,\zeta,\tau)$ the inverse transformation to \eqref{RPsi}. Then, 
\begin{align*}
\tilde \Lambda_1(R,\Theta,\zeta,\tau)  
\equiv   &  \, \mathcal L 
 \tilde R(r,\theta,\zeta,\tau) \Big|_{r = \tilde r(R,\Theta,\zeta,\tau), \theta = \tilde \theta(R,\Theta,\zeta,\tau)} -\sum_{K\in\mathcal X} \tau^{-\frac{K}{B}} \Lambda_{1,K}(\tilde r(R,\Theta,\zeta,\tau),\tilde \theta(R,\Theta,\zeta,\tau),\tau), \\
\tilde \Lambda_2(R,\Theta,\zeta,\tau)  
\equiv   &  \, \mathcal L 
 \tilde \Theta(r,\theta,\zeta,\tau) \Big|_{r = \tilde r(R,\Theta,\zeta,\tau), \theta = \tilde \theta(R,\Theta,\zeta,\tau)} -\sum_{K\in\mathcal X} \tau^{-\frac{K}{B}} \Lambda_{2,K}(\tilde r(R,\Theta,\zeta,\tau),\tilde \theta(R,\Theta,\zeta,\tau),\tau), \\
\eta_1(R,\Theta,\zeta,\tau) \equiv & \, \left( \xi_1(r,\theta,\zeta,\tau) \partial_r  +\tau^{-\frac{A}{B}} \xi_2(r,\theta,\zeta,\tau) \partial_\theta\right) \tilde R(r,\theta,\zeta,\tau)\Big|_{r = \tilde r(R,\Theta,\zeta,\tau), \theta = \tilde \theta(R,\Theta,\zeta,\tau)}, \\
\eta_2(R,\Theta,\zeta,\tau) \equiv & \, \left(\tau^{ \frac{A}{B}}  \xi_1 (r,\theta,\zeta,\tau) \partial_r  +\xi_2 (r,\theta,\zeta,\tau)\partial_\theta\right) \tilde\Theta(r,\theta,\zeta,\tau)\Big|_{r = \tilde r(R,\Theta,\zeta,\tau), \theta = \tilde \theta(R,\Theta,\zeta,\tau)}.
\end{align*}
It follows that for every $i\in\{1,2\}$ the following estimates hold:
\begin{align*}
&\tilde\Lambda_i(R,\Theta,\zeta,\tau)=\mathcal O(\tau^{-\frac{3A}{B}}), \quad
\eta_i(R,\Theta,\zeta,\tau)=  \xi_{i}(R,\Theta,\zeta,\tau)+\mathcal O(\tau^{-\frac{A}{B}})
\end{align*}
as $\tau\to\infty$ uniformly for all $|R|\leq r_0-\varepsilon$ and $(\Theta,\zeta)\in\mathbb R^2$.

Thus, we obtain the proof of Theorem~\ref{Th1} with
$\tilde Z_1(R,\Theta,\zeta,\tau)\equiv (\tilde r(R,\Theta,\zeta,\tau)-R)\tau^{A/B}$, $ \tilde Z_2(R,\Theta,\zeta,\tau)\equiv (\tilde \theta(R,\Theta,\zeta,\tau)-\Theta)\tau^{A/B}$, $\tilde{t}_0=(B\tilde\tau_0)^{1/B}$and $R_0=r_0-\varepsilon$.

\section{Asymptotic regimes of the truncated system}
\label{Sec6}
\begin{proof}[Proof of Lemma~\ref{LTh2}]
Consider the system
\begin{gather}\label{sysrtheta}
\begin{cases} 
\hat\lambda_1(R,\Theta,\tau)=0, \\
\hat\lambda_2(R,\Theta,\tau)=0.
\end{cases}
\end{gather}
From \eqref{lambda12} and \eqref{astheta} it follows that $\tau^{ A/B}\hat\lambda_i(0,\Theta_0,\tau)=o(1)$ and 
\begin{gather*}
\tau^{ \frac{2A}{B}}\begin{vmatrix}
\partial_R\hat\lambda_1(0,\Theta_0,\tau) & \partial_\Theta\hat\lambda_1(0,\Theta_0,\tau) \\
\partial_R\hat\lambda_2(0,\Theta_0,\tau) & \partial_\Theta\hat\lambda_2(0,\Theta_0,\tau)
\end{vmatrix} = - \chi_{2,A,0}\mathcal P'(\Theta_0)+o(1), \quad \tau\to\infty.
\end{gather*}
Hence, there exists $\tau_0\geq \tilde \tau_0$ such that for all $\tau\geq \tau_0$ system \eqref{sysrtheta} has a smooth solution $\hat R (\tau)$, $\hat \Theta (\tau)$ such that $\hat R(\tau) \to 0$ and $\hat \Theta (\tau) \to \Theta_0$ as $\tau\to\infty$. It can easily be checked that
\begin{gather*}
\hat R(\tau)=-\tau^{-\frac{A}{B}}\left(\frac{\vartheta_{2,2A,0}(\Theta_0)}{\chi_{2,2A,0}}+\mathcal O(\tau^{-\frac{\beta}{B h}})\right), \quad 
\hat\Theta(\tau)=\Theta_0+\mathcal O(\tau^{-\frac{\beta}{B h}}), \quad \tau\to\infty.
\end{gather*}
Define the parameter
\begin{gather*}
\ell=\begin{cases}
\displaystyle \frac{A}{2B}, & \text{if}\quad B>2A, \\
\displaystyle  \min\left\{\frac{A}{2B},\frac{|J(\Theta_0)|}{3}\right\}, &\text{if}\quad B=2A.
\end{cases}
\end{gather*}
Then, substituting
\begin{gather}\label{subs111}
R(\tau)=\hat R(\tau)+\tau^{-\ell}z_1(\tau), \quad \Theta(\tau)=\hat \Theta(\tau)+\tau^{-\ell} z_2(\tau)
\end{gather}
into \eqref{det1}, we obtain a near-Hamiltonian system
\begin{gather}\label{nHam}
	\begin{split}
		\frac{dz_1}{d\tau}= &\, -  \tau^{-\frac{A}{B}}\partial_{z_2} \mathcal H(z_1,z_2,\tau)+\tau^{-\frac{2A}{B}} \mathcal G_1(z_1,z_2,\tau),\\
		\frac{dz_2}{d\tau}= & \, \tau^{-\frac{A}{B}}\partial_{z_1} \mathcal H(z_1,z_2,\tau)+\tau^{-\frac{2A}{B}}\left(\mathcal J(z_1,z_2,\tau)+ \mathcal G_2(z_1,z_2,\tau)\right),
	\end{split}
\end{gather}
where
\begin{gather}\label{HJGG}
	\begin{split}
\mathcal H(z_1,z_2,\tau)\equiv & \, \tau^{\ell+\frac{ A}{B}} \int\limits_0^{z_1} \hat \lambda_2(\hat R(\tau)+\tau^{-\ell}\varsigma_1,\hat \Theta(\tau),\tau)\,d\varsigma_1 \\ & -  \tau^{\ell+\frac{A}{B}} \int\limits_0^{z_2} \hat \lambda_1(\hat R(\tau)+\tau^{-\ell}z_1,\hat \Theta(\tau)+\tau^{-\ell}\varsigma_2,\tau)\,d\varsigma_2,\\
\mathcal J(z_1,z_2,\tau)\equiv & \, \tau^{\ell+ \frac{2A}{B}}\sum_{j=1}^2 \int\limits_0^{z_2} \partial_{z_i}\hat \lambda_i(\hat R(\tau)+\tau^{-\ell}z_1,\hat \Theta(\tau)+\tau^{-\ell}\varsigma_2,\tau) \,d\varsigma_2+\tau^{\frac{2A}{B}-1}2\ell z_2,\\
\mathcal G_1(z_1,z_2,\tau)\equiv & \, \tau^{\ell+ \frac{2A}{B}} \left(\tilde \lambda_1(\hat R(\tau)+\tau^{-\ell}z_1,\hat \Theta(\tau)+\tau^{-\ell}z_2,\zeta(\tau),\tau)-\hat R'(\tau)\right), \\
\mathcal G_2(z_1,z_2,\tau)\equiv & \, \tau^{\ell+ \frac{2A}{B}} \left(\tilde \lambda_2(\hat R(\tau)+\tau^{-\ell} z_1,\hat \Theta(\tau)+\tau^{-\ell} z_2,\zeta(\tau),\tau)-\hat \Theta'(\tau)\right).
	\end{split}
\end{gather}
Note that $\mathcal H(0,0,\tau)\equiv0$, $\partial_{z_i}\mathcal H(0,0,\tau)\equiv 0$, and $\mathcal J(0,0,\tau)\equiv 0$, while $\mathcal G_i(0,0,\tau)\not\equiv 0$ for $i\in\{1,2\}$. Moreover, 
\begin{gather}\label{estHJ}
\begin{split}
\mathcal H(z_1,z_2,\tau)=&\frac{1}{2}\left(\chi_{2,A,0} z_1^2-\mathcal P'(\Theta_0)z_2^2\right) +\mathcal O(\tau^{-\kappa})\mathcal O(|{\bf z}|^2),\\
\mathcal J(z_1,z_2,\tau)=&z_2 \left(J(\Theta_0)+\delta_{2A,B}2\ell+\mathcal O(\tau^{-\kappa})\right),
\end{split}
\end{gather}
and $\mathcal G_i(z_1,z_2,\tau) =  \mathcal O(\tau^{-\kappa})$ as $\tau\to\infty$ and $|{\bf z}|:=\sqrt{z_1^2+z_2^2}\to 0$ with $\kappa=\min\{A/(2B),\beta/(hB),\ell\}>0$. Note that the functions $\mathcal G_1(z_1,z_2,\tau)$ and $\mathcal G_2(z_1,z_2,\tau)$ can be considered as non-vanishing perturbations of the system with equilibrium at $(z_1,z_2)=(0,0)$. 

The rest of the proof is divided into three parts.

{\bf 1}. Let $\mathcal P'(\Theta_0)<0$ and $J(\Theta_0)<0$. In this case, the function $\mathcal H(z_1,z_2,\tau)$ is positive definite as $|{\bf z}|\to 0$ and $\tau\to \infty$. Hence, the equilibrium $(0,0)$ of the corresponding limiting system 
\begin{gather}\label{limzz}
\tau^{\frac{A}{B}}\frac{d\tilde z_1}{d\tau}=\mathcal P'(\Theta_0)\tilde z_2, 
\quad 
\tau^{\frac{A}{B}}\frac{d\tilde z_2}{d\tau}=\chi_{2,A,0}\tilde z_1 
\end{gather}
is a center. Let us prove the stability of the equilibrium with respect to perturbations $\mathcal G_i(R,\Theta,\tau)$. Consider the Lyapunov function candidate for system \eqref{nHam} in the following form:
\begin{gather}\label{LFC}
\mathcal V(z_1,z_2,\tau)\equiv \mathcal H(z_1,z_2,\tau)+\tau^{-\frac{A}{B}} \frac{\tilde Jz_1z_2}{2}  
\end{gather} 
 with $\tilde J=J(\Theta_0)+\delta_{2A,B}2\ell\leq -|J(\Theta_0)|/3<0$.
The derivative of $\mathcal V(z_1,z_2,\tau)$ along the trajectories of system \eqref{nHam} is given by
\begin{gather}
\label{dVest}
\begin{split}
\frac{d\mathcal V}{d\tau}\Big|_\eqref{nHam}:= &  \partial_\tau \mathcal V + \partial_{z_1}\mathcal V \left(-  \tau^{-\frac{A}{B}}\partial_{z_2} \mathcal H +\tau^{-\frac{2A}{B}} \mathcal G_1 \right) + \partial_{z_2}\mathcal V \left( \tau^{-\frac{A}{B}}\partial_{z_1} \mathcal H +\tau^{-\frac{2A}{B}} \left(\mathcal J+  \mathcal G_2 \right)\right)\\
=&\tau^{-\frac{2A}{B}}\left(\frac{\tilde J}{2}\left(\chi_{2,A,0} z_1^2-\mathcal P'(\Theta_0)z_2^2\right) +\mathcal O(\tau^{-\kappa})\mathcal O(|{\bf z}|)\right)
\end{split}
\end{gather}
as $\tau\to\infty$ and $|{\bf z}|\to 0$. It follows from \eqref{estHJ}, \eqref{LFC} and \eqref{dVest} that
there exist  $\tau_1\geq \tau_0$, $\Delta_1>0$,  $\mathcal V_-=\min\{\chi_{2,A,0}, |\mathcal P'(\Theta_0)|\}/4>0$, $\mathcal V_+=\max\{\chi_{2,A,0}, |\mathcal P'(\Theta_0)|\}>0$, $D_1=|\tilde J|\mathcal V_-$ and $D_2>0$ such that 
\begin{gather*}
\mathcal V_- |{\bf z}|^2\leq \mathcal V(z_1,z_2,\tau)\leq \mathcal V_+|{\bf z}|^2, \quad 
\frac{d\mathcal V}{d\tau}\Big|_{\eqref{nHam}}\leq \tau^{-\frac{2A}{B}}\left(-D_1 |{\bf z}|^2+\tau^{-\kappa}D_2|{\bf z}|\right)  
\end{gather*}
for all $\tau\geq \tau_1$ and $|{\bf z}|\leq \Delta_1$. Therefore, for all $\varepsilon>0$ there exist
\begin{gather*}
\tau_\varepsilon=\max\left\{\left(\frac{4D_2}{D_1 \varepsilon}\sqrt{\frac{\mathcal V_+}{\mathcal V_-}}\right)^{\frac 1\kappa}, \tau_1\right\}, \quad \delta_\varepsilon= \frac{2D_2 \tau_\varepsilon^{-\kappa}}{D_1}
\end{gather*}
such that $d\mathcal V/d\tau\leq -\tau^{-2A/B}D_1 |{\bf z}|^2/2$ for all $\tau\geq \tau_\varepsilon$ and $\delta_\varepsilon\leq |{\bf z}|\leq \varepsilon$. Combining this with the inequalities
\begin{gather*}
\sup_{|{\bf z}|\leq \delta_\varepsilon} \mathcal V(z_1,z_2,\tau)\leq \mathcal V_+ \delta_\varepsilon^2<\mathcal V_- \varepsilon^2=\inf_{|{\bf z}|=\varepsilon}\mathcal V(z_1,z_2,\tau)
\end{gather*}
for all $\tau\geq \tau_\varepsilon$, we see that any solution $(z_1(\tau),z_2(\tau))$ of system \eqref{nHam} with initial data $|{\bf z}(\tau_\varepsilon)|\leq \delta_\varepsilon$ satisfies $|{\bf z}(\tau)|<\varepsilon$ for all $\tau>\tau_\varepsilon$. From \eqref{subs111} and  the continuity of solutions to the Cauchy problem with respect to the initial data it follows that there exists a solution $R_\ast(\tau)$, $\Theta_\ast(\tau)$ of system \eqref{det1} defined for all $\tau\geq \tau_0$ such that $R_\ast(\tau)=\hat R(\tau)+\mathcal O(\tau^{-\ell})$ and $\Theta_\ast(\tau)= \hat\Theta(\tau)+\mathcal O(\tau^{-\ell})$ as $\tau\to\infty$.

Let us show that the solution $R_\ast(\tau)$, $\Psi_\ast(\tau)$ is stable in system \eqref{det1}. Substituting $R(\tau)=R_\ast(\tau)+\tau^{-\ell} z_1(\tau)$, $\Theta(\tau)=\Theta_\ast(\tau)+\tau^{-\ell}z_2(\tau)$ into \eqref{det1}, we obtain \eqref{nHam}, where the functions $\mathcal H(z_1,z_2,\tau)$, $\mathcal J(z_1,z_2,\tau)$ and $\mathcal G_i(z_1,z_2,\tau)$ defined by \eqref{HJGG} with $R_\ast(\tau)$, $\Psi_\ast(\tau)$ instead of $\hat R(\tau)$, $\hat \Psi(\tau)$. In this case, $\mathcal H(0,0,\tau)\equiv \mathcal J(0,0,\tau)\equiv \mathcal G_i(0,0,\tau)\equiv 0$. Moreover, the estimates \eqref{estHJ} hold and $\mathcal G_i(z_1,z_2,\tau)=\mathcal O(|{\bf z}|)\mathcal O(\tau^{-A/B})$ as $\tau\to\infty$ and $|{\bf z}|\to 0$. Using \eqref{LFC} as a Lyapunov function candidate, we see that
\begin{gather*}
\frac{d\mathcal V}{d\tau}\Big|_\eqref{nHam} =\tau^{-\frac{2A}{B}}\left(\frac{\tilde J}{2}\left(\chi_{2,A,0} z_1^2-\mathcal P'(\Theta_0)z_2^2\right) +\mathcal O(\tau^{-\kappa})\mathcal O(|{\bf z}|^2)\right)
\end{gather*}
as $\tau\to\infty$ and $|{\bf z}|\to 0$. Hence, there exist $\tau_\ast\geq \tau_1$ and $0<\Delta_\ast\leq \Delta_1$ such that 
\begin{gather}\label{Vineq2}
\frac{d\mathcal V}{d\tau}\Big|_\eqref{nHam}\leq -\tau^{-\frac{2A}{B}}\ell_\ast\mathcal V 
\end{gather}
for all $\tau\geq \tau_\ast$ and $|{\bf z}|\leq \Delta_\ast$ with $\ell_\ast=|\tilde J|\mathcal V_-/(2\mathcal V_+)>0$. Therefore, for all $\varepsilon\in (0,\Delta_\ast)$ there exists $\delta\in (0,\varepsilon)$ such that any solution of system \eqref{nHam} with initial data $|{\bf z}(\tau_\ast)|\leq \delta$ cannot leave the domain $\{|{\bf z}|\leq \varepsilon\}$ as $\tau>\tau_\ast$. Moreover,  integrating \eqref{Vineq2} with respect to $\tau$, we obtain
\begin{align*}
\mathcal V(z_1(\tau),z_2(\tau),\tau)\leq &  \mathcal V(z_1(\tau_\ast),z_2(\tau_\ast),\tau_\ast) \exp\left(-\frac{B \ell_\ast}{B-2A}\left(\tau^{1-\frac{2A}{B}}-\tau_\ast^{1-\frac{2A}{B}}\right)\right)  & \text{if}&  \ \ B>2A,\\
\mathcal V(z_1(\tau),z_2(\tau),\tau)\leq & \mathcal V(z_1(\tau_\ast),z_2(\tau_\ast),\tau_\ast)\left(\frac{\tau}{\tau_\ast}\right)^{-\ell_\ast} & \text{if}& \ \  B=2A
\end{align*}
as $\tau\geq \tau_\ast$. Thus, the solution $R_\ast(\tau)$, $\Psi_\ast(\tau)$ is asymptotically stable.

{\bf 2}. Consider now the case $\mathcal P'(\Theta_0)<0$ and $J(\Theta_0)>0$. Substituting
\begin{gather*}
R(\tau)=\hat R(\tau)+z_1(\tau), \quad \Theta(\tau)=\hat \Theta(\tau)+ z_2(\tau)
\end{gather*}
into \eqref{det1}, we obtain
\begin{gather}\label{nHam2}
	\begin{split}
		\frac{dz_1}{d\tau}= &\, -  \tau^{-\frac{A}{B}}\partial_{z_2} \tilde{\mathcal H}(z_1,z_2,\tau)+\tau^{-\frac{2A}{B}} \tilde{\mathcal G}_1(z_1,z_2,\tau),\\
		\frac{dz_2}{d\tau}= & \, \tau^{-\frac{A}{B}}\partial_{z_1} \tilde{\mathcal H}(z_1,z_2,\tau)+\tau^{-\frac{2A}{B}} \tilde{\mathcal J}(z_1,z_2,\tau)+ \tau^{-\frac{3A}{B}}\tilde{\mathcal G}_2(z_1,z_2,\tau) ,
	\end{split}
\end{gather}
where
\begin{gather*}
	\begin{split}
\tilde{\mathcal H}(z_1,z_2,\tau)\equiv & \, \tau^{\frac{ A}{B}} \int\limits_0^{z_1} \hat \lambda_2(\hat R(\tau)+\varsigma_1,\hat \Theta(\tau),\tau)\,d\varsigma_1 -  \tau^{\frac{A}{B}} \int\limits_0^{z_2} \hat \lambda_1(\hat R(\tau)+z_1,\hat \Theta(\tau)+\varsigma_2,\tau)\,d\varsigma_2,\\
\tilde{\mathcal J}(z_1,z_2,\tau)\equiv & \, \tau^{\frac{2A}{B}}\sum_{j=1}^2 \int\limits_0^{z_2} \partial_{z_i}\hat \lambda_i(\hat R(\tau)+  z_1,\hat \Theta(\tau)+ \varsigma_2,\tau) \,d\varsigma_2,\\
\tilde{\mathcal G}_1(z_1,z_2,\tau)\equiv & \, \tau^{\frac{3A}{B}} \left(\tilde \lambda_1(\hat R(\tau)+z_1,\hat \Theta(\tau)+z_2,\zeta(\tau),\tau)-\hat R'(\tau)\right), \\
\tilde{\mathcal G}_2(z_1,z_2,\tau)\equiv & \, \tau^{\frac{3A}{B}} \left(\tilde \lambda_2(\hat R(\tau)+ z_1,\hat \Theta(\tau)+ z_2,\zeta(\tau),\tau)-\hat \Theta'(\tau)\right).
	\end{split}
\end{gather*}
Note that 
\begin{gather}\label{estHJ2}
\begin{split}
\tilde{\mathcal H}(z_1,z_2,\tau)=&\frac{1}{2}\left(\chi_{2,A,0} z_1^2-\mathcal P'(\Theta_0)z_2^2\right) +\mathcal O(|{\bf z}|^3)+\mathcal O(\tau^{-\tilde\kappa})\mathcal O(|{\bf z}|^2),\\
\tilde{\mathcal J}(z_1,z_2,\tau)=&z_2 \left(J(\Theta_0) +\mathcal O(|{\bf z}|^3)+\mathcal O(\tau^{-\tilde\kappa})\right),
\end{split}
\end{gather}
and $\tilde{\mathcal G}_i(z_1,z_2,\tau) =  \mathcal O(1)$ as $\tau\to\infty$ and $|{\bf z}|\to 0$ with $\tilde\kappa=\min\{A,\beta/h\}/B>0$.
Consider a  function  
\begin{gather}\label{LFC2}
\mathcal W(z_1,z_2,\tau)\equiv \sqrt{\tilde{\mathcal V}(z_1,z_2,\tau)}, \quad \tilde{\mathcal V}(z_1,z_2,\tau)\equiv \tilde{\mathcal H}(z_1,z_2,\tau)+\tau^{-\frac{A}{B}} \frac{J(\Theta_0) z_1z_2}{2}.
\end{gather} 
The derivative of $ \mathcal W(z_1,z_2,\tau)$ along the trajectories of system \eqref{nHam2} is given by
\begin{gather}
\label{dVest2}
\begin{split}
\frac{d\mathcal W}{d\tau}\Big|_\eqref{nHam2} =\tau^{-\frac{2A}{B}}\frac{1}{2\sqrt{\tilde{\mathcal V}}}\left(\frac{ J(\Theta_0)}{2}\left(\chi_{2,A,0} z_1^2-\mathcal P'(\Theta_0)z_2^2\right)+\mathcal O(|{\bf z}|^3) +\mathcal O(\tau^{-\tilde\kappa})\mathcal O(|{\bf z}|)\right)
\end{split}
\end{gather}
as $\tau\to\infty$ and $|{\bf z}|\to 0$. It follows from \eqref{estHJ2}, \eqref{LFC2} and \eqref{dVest2} that there exist  $\tilde\tau_1\geq \tau_0$, $\tilde\Delta_1>0$,  $\mathcal W_\pm =\sqrt{\mathcal V_\pm}$, $\tilde D_1=|J(\Theta_0)|\mathcal W_-/(2\mathcal W_+)>0$ and $\tilde D_2>0$ such that 
\begin{gather*}
\mathcal W_- |{\bf z}| \leq \mathcal W(z_1,z_2,\tau)\leq \mathcal W_+|{\bf z}|, \quad 
\frac{d\mathcal W}{d\tau}\Big|_{\eqref{nHam2}}\geq \tau^{-\frac{2A}{B}}\left(\tilde D_1 \mathcal W-\tau^{-\tilde\kappa}\tilde D_2 \right)  
\end{gather*}
for all $\tau\geq \tilde\tau_1$ and $|{\bf z}|\leq \tilde\Delta_1$. Integrating the last inequality as $\tau\geq\tilde \tau_\ast$ with some $\tau_\ast \geq\tilde \tau_1$, we obtain
\begin{gather*}
\mathcal W(z_1(\tau),z_2(\tau),\tau)\geq e^{\tilde D(\tau)}\left(\mathcal W(z_1(\tilde\tau_\ast),z_2(\tilde\tau_\ast),\tilde\tau_\ast)e^{-\tilde D(\tilde\tau_\ast)}-\frac{\tilde D_2}{\tilde D_1}\int\limits_{\tilde\tau_\ast}^\tau \varsigma^{-\tilde\kappa}\tilde D'(\varsigma)e^{-\tilde D(\varsigma)}\,d\varsigma\right) 
\end{gather*}
 for solutions of system \eqref{nHam2} lying in $\{(z_1,z_1)\in\mathbb R^2:|{\bf z}|\leq \tilde\Delta_1\}$. Here, $\tilde D'(\tau)\equiv D_1\tau^{- {2A}/{B}}$. Hence, for all $\tilde \delta\in (0,\tilde \Delta_1)$ there exists $\tilde \tau_\ast= \max\{\tilde\tau_1,(2\tilde D_2/(\tilde \delta \mathcal W_-\tilde D_1))^{1/\tilde\kappa}\}$ such that $\mathcal W(z_1(\tau),z_2(\tau),\tau)\geq \tilde \delta e^{\tilde D(\tau)-\tilde D(\tilde \tau_\ast)}/2$ as $\tau\geq\tilde\tau_\ast$ for solutions of system \eqref{nHam2} with initial data $|{\bf z}(\tilde \tau_\ast)|=\tilde \delta$. Since $\tilde D(\tau)$ is strictly increasing, we see that there exists $\tilde\tau_e>\tilde \tau_\ast$ such that $|{\bf z}(\tilde \tau_e)|=\tilde\Delta_1$. Returning to the variables $(R,\Theta)$, we obtain the instability of the asymptotic regime corresponding to $\hat R(\tau)$, $\hat \Theta(\tau)$.

Finally, let $\mathcal P'(\Theta_0)>0$. In this case, the equilibrium $(0,0)$ of system \eqref{limzz} is a saddle. Consider 
$\mathcal U(z_1,z_2)\equiv z_1z_2$ as a Chetaev function candidate for system \eqref{nHam2}. It can easily be checked that
\begin{gather*}
\frac{d\mathcal U}{d\tau}\Big|_{\eqref{nHam2}}=\tau^{-\frac{A}{B}}\left(\chi_{2,A,0} z_1^2+\mathcal P'(\Theta_0)z_2^2 +\mathcal O(|{\bf z}|^3) +\mathcal O(\tau^{-\tilde\kappa})\mathcal O(|{\bf z}|)\right)
\end{gather*}
as $\tau\to\infty$ and $|{\bf z}|\to 0$. Hence, there exists $\hat \tau_1>\tau_0$ and $\hat D_0>0$ such that 
\begin{gather*}
\frac{d\mathcal U}{d\tau}\Big|_{\eqref{nHam2}}\geq \tau^{-\frac{A}{B}}\left(\mathcal V_- |{\bf z}|^2-\tau^{-\tilde\kappa}\hat D_0 |{\bf z}|\right)
\end{gather*}
for all $\tau\geq \hat\tau_1$ and $|{\bf z}|\leq \hat\Delta_1$. Hence, for all $\hat \delta>0$ there exists $\hat \tau_\ast=\max\{\hat \tau_1, (2\hat D_0/(\hat\delta\mathcal V_-))^{1/\tilde\kappa}\}$ such that $d\mathcal U/d\tau\geq \tau^{-A/B} \mathcal V_- \hat\delta^2 /2>0$ for all $\hat\delta\leq |{\bf z}|\leq \hat\Delta_1$ and $\tau\geq \hat\tau_\ast$. Integrating the last inequality yields
\begin{gather}\label{Uineq}
\mathcal U(z_1(\tau),z_2(\tau))\geq \mathcal U(z_1(\tilde \tau_\ast),z_2(\tilde \tau_\ast))+\frac{\hat\delta^2 B \mathcal V_- }{2(B-A)} \left(\tau^{1-\frac{A}{B}}-\hat\tau_\ast^{1-\frac{A}{B}}\right)
\end{gather} 
for solutions of system \eqref{nHam2} lying in $
\mathfrak D:=\{(z_1,z_2)\in\mathbb R^2:\hat\delta\leq |{\bf z}|\leq \hat\Delta_1\}$. Consider a solution of system \eqref{nHam2} with initial data $\mathcal U(z_1(\tilde\tau_\ast),z_2(\tilde\tau_\ast))=\tilde\delta^2>0$. Then, the inequality \eqref{Uineq} shows that the solution cannot stay forever in $\{(z_1,z_2)\in\mathfrak D: z_1z_2>0\}$ because $\mathcal U(z_1,z_2)$ is bounded on $\mathfrak D$. Since $\mathcal U(z_1,z_2)\leq |{\bf z}|^2/2$ for all $(z_1,z_2)\in\mathbb R^2$, it follows that there exists $\hat \tau_e>\hat\tau_\ast$ such that $|{\bf z}(\hat\tau_e)|=\hat\Delta_1$. 
\end{proof} 

\begin{proof}[Proof of Lemma~\ref{LTh3}]
It follows from \eqref{asthetanot} that there are positive parameters $D_1$, $D_2$, $D_3$ and $\tau_1\geq \tilde \tau_0$ such that
\begin{gather*}
\left|\frac{dR}{d\tau}\right|\geq \tau^{-\frac{A}{B}} D_1, \quad 
\left|\frac{d\Theta}{d\tau}\right|\geq \tau^{-\frac{A}{B}} \left(D_2|R|-D_3\right)
\end{gather*}
for all $|R|\leq R_0$, $\Theta\in\mathbb R$ and $\tau\geq \tau_1$. It follows that 
\begin{align*}
&|R(\tau)-R(\tau_1)|\geq \tilde D_1 \left(\tau^{\frac{B-A}{B}}-\tau_1^{ \frac{B-A}{B}}\right), \\
&|\Theta(\tau)-\Theta(\tau_1)|\geq \tilde D_2 \left(\tau^{ \frac{2(B-A)}{B} }-\tau_1^{\frac{2(B-A)}{B}}\right)-\tilde D_3 \left(\tau^{ \frac{B-A}{B}}-\tau_1^{ \frac{B-A}{B}}\right)
\end{align*}
as $\tau\geq \tau_1$, where $\tilde D_1=B D_1/(B-A)>0$, $\tilde D_2=B \tilde D_1 D_2 /(2(B-A))>0$ and $\tilde D_3=B(D_3+|R(\tau_1)|-\tilde D_1 D_2 \tau_1^{1-A/B})/(B-A)$. Thus, the solutions of system \eqref{det1} grow unbounded. Hence, there exist $\tau_e>\tau_1$ such that $|R(\tau_e)|+|\Theta(\tau_e)|>\Delta$.
\end{proof}

\section{Stochastic stability of the resonance}
\label{Sec7}
\begin{proof}[Proof of Theorem~\ref{Th2}]
Substituting 
\begin{gather*}
R(\tau)=R_\ast(\tau)+z_1(\tau), \quad 
\Theta(\tau)=\Theta_\ast(\tau)+z_1(\tau)
\end{gather*}
into \eqref{veq}, we obtain
\begin{gather}\label{nHamS}
	\begin{split}
		dz_1= &  -  \tau^{-\frac{A}{B}}\partial_{z_2}  \mathcal H (z_1,z_2,\tau)\,d\tau + \tau^{-\frac{C-A}{B}}\mu \mathcal E_1(z_1,z_2,\tau)\,d\tilde w(\tau),\\
		dz_2= &     \left(
		\tau^{-\frac{A}{B}}\partial_{z_1}  \mathcal H (z_1,z_2,\tau)+
		\tau^{-\frac{2A}{B}}\mathcal J (z_1,z_2,\tau)\right) d\tau + \tau^{-\frac{C}{B}}\mu \mathcal E_2(z_1,z_2,\tau)\,d\tilde w(\tau),
	\end{split}
\end{gather}
where
\begin{gather*}
	\begin{split}
\mathcal H(z_1,z_2,\tau)\equiv & \, \int\limits_0^{z_1}\Big( \Lambda_2(R_\ast(\tau)+\varsigma_1,\Theta_\ast(\tau),\zeta(\tau),\tau)-\Lambda_2(R_\ast(\tau),\Theta_\ast(\tau),\zeta(\tau),\tau)\Big)\,d\varsigma_1 \\ & -    \int\limits_0^{z_2} \Big(\Lambda_1(R_\ast(\tau)+z_1,\Theta_\ast(\tau)+\varsigma_2,\zeta(\tau),\tau)-\Lambda_1(R_\ast(\tau),\Theta_\ast(\tau),\zeta(\tau),\tau)\Big)\,d\varsigma_2,\\
\mathcal J(z_1,z_2,\tau)\equiv & \, \sum_{j=1}^2 \int\limits_0^{z_2} \partial_{z_i}\Lambda_i(R_\ast(\tau)+ z_1,\Theta_\ast(\tau)+ \varsigma_2,\zeta(\tau),\tau) \,d\varsigma_2,\\
 \mathcal E_i(z_1,z_2,\tau)\equiv & \, \eta_i(R_\ast(\tau)+z_1,\Theta_\ast(\tau)+z_2,\zeta(\tau),\tau).
	\end{split}
\end{gather*}
Note that $\mathcal H(z_1,z_2,\tau)\equiv 0$, $\partial_{z_i}\mathcal H(z_1,z_2,\tau)\equiv 0$ and $\mathcal J(0,0,\tau)\equiv 0$. It follows from \eqref{lambda12} and \eqref{asthetachi} that
\begin{gather*}
\begin{split}
\mathcal H(z_1,z_2,\tau)=&\,\frac{1}{2}\left(\chi_{2,A,0} z_1^2-\mathcal P'(\Theta_0)z_2^2\right)+\mathcal O(|{\bf z}|^3) +\mathcal O(\tau^{-\kappa})\mathcal O(|{\bf z}|^2),\\
\mathcal J(z_1,z_2,\tau)=&\,z_2 \left(J(\Theta_0) +\mathcal O(\tau^{-\kappa})\right), \\
\mathcal E_i(z_1,z_2,\tau) = &\, \mathcal O(1)
\end{split}
\end{gather*}
as $\tau\to\infty$ and $|{\bf z}| \to 0$ with $\kappa=\min\{A,\beta/h\}/B>0$. If $\mu=0$, system \eqref{nHamS} has an equilibrium point at the origin. Let us prove the stability of the equilibrium with respect to white noise perturbations by constructing the stochastic Lyapunov function~\cite{HJK67,RH12}. 

The generator of the process defined by \eqref{nHamS} has the  form $\mathfrak L\equiv  \mathfrak L_0+\mu^2 \mathfrak L_1$, where
\begin{align*}
\mathfrak L_0:=& \partial_\tau-\tau^{-\frac{A}{B}}\partial_{z_2}\mathcal H \partial_{z_1}+\left( \tau^{-\frac{A}{B}}\partial_{z_1}\mathcal H+\tau^{-\frac{2A}{B}}\mathcal J\right)\partial_{z_2},\\
\mathfrak L_1:=& \frac{1}{2}\tau^{-\frac{2(C-A)}{B}}\left( \mathcal E_1^2\partial_{z_1}^2+2\tau^{-\frac{ A}{B}}\mathcal E_1 \mathcal E_2\partial_{z_1}\partial_{z_2}+\tau^{-\frac{2A}{B}}\mathcal E_2^2\partial_{z_2}^2\right).
\end{align*}
Consider the auxiliary function
\begin{gather*}
\mathcal V_0(z_1,z_2,\tau)\equiv \mathcal H(z_1,z_2,\tau)+\tau^{-\frac{A}{B}} \frac{J(\Theta_0) z_1z_2}{2}.
\end{gather*}
It can easily be checked that
\begin{align*}
&\mathfrak L_0 \mathcal V_0(z_1,z_2,\tau) =\tau^{-\frac{2A}{B}}\left(\frac{J(\Theta_0)}{2}\left(\chi_{2,A,0} z_1^2-\mathcal P'(\Theta_0)z_2^2\right) +\mathcal O(|{\bf z}|^3)+\mathcal O(\tau^{-\kappa})\mathcal O(|{\bf z}|^2)\right), \\
&\mathfrak L_1 \mathcal V_0(z_1,z_2,\tau) = \mathcal O(\tau^{-\frac{2(C-A)}{B}})
\end{align*}
as $\tau\to\infty$ and $|{\bf z}|\to 0$. Hence, there exist 
$\tau_1>\tau_0$, 
$r_\ast>0$, 
$\mathcal V_+=\max\{\chi_{2,A,0},|\mathcal P'(\Theta_0)|\}>0$, 
$\mathcal V_-=\min\{\chi_{2,A,0},|\mathcal P'(\Theta_0)|\}>0$, 
$D_0=|J|\mathcal V_->0$ and $D_1>0$ such that
\begin{gather}\label{V0ests}
\begin{split}
&\mathcal V_- |{\bf z}|^2\leq \mathcal V_0(z_1,z_2,\tau) \leq \mathcal V_+ |{\bf z}|^2, \\
&\mathfrak L \mathcal V_0(z_1,z_2,\tau)\leq - \tau^{-\frac{2A}{B}} D_0 |{\bf z}|^2 + \mu^2 D_1 \tau^{-\frac{2(C-A)}{B}} 
\end{split}
\end{gather}
for all $\tau\geq \tau_1$ and $|{\bf z}|\leq r_\ast$.

Fix the parameters $\varepsilon_1\in (0,r_\ast)$, $\varepsilon_2\in (0,1)$ and $\tau_\ast\geq \tau_1$. Consider the stochastic Lyapunov function candidate for system \eqref{nHamS} in the form
\begin{gather}\label{SLF1}
\mathcal V(z_1,z_2,\tau)\equiv \mathcal V_0(z_1,z_2,\tau)+\mu^2 \mathcal V_1(\tau),
\end{gather}
where
\begin{gather*}
\mathcal V_1(\tau)\equiv 
\begin{cases}
\displaystyle D_1 \tau_\ast^{-\frac{2(C-A)}{B}}(\mathcal T+\tau_\ast-\tau), &\text{if}\quad C<A+{B}/{2},\\ 
\displaystyle D_1 \log\left(\frac{\mathcal T+\tau_\ast}{\tau}\right), & \text{if}\quad  C =A+{B}/{2},\\ 
\displaystyle D_1 \int\limits_\tau^{\tau_\ast+\mathcal T}\varsigma^{-\frac{2(C-A)}{B}}\,d\varsigma, & \text{if}\quad    C > A+{B}/{2} 
\end{cases}
\end{gather*}
with some $\mathcal T>0$. It can easily be checked that 
\begin{gather}
\label{V1est}
\mathcal V(z_1,z_2,\tau)\geq \mathcal V_0(z_1,z_2,\tau)\geq 0, \quad 
\mathfrak L \mathcal V(z_1,z_2,\tau)\leq - \tau^{-\frac{2A}{B}} D_0 |{\bf z}|^2\leq 0
\end{gather}
 for all $({\bf z},\tau)\in \mathfrak D(r_\ast,\tau_\ast,\mathcal T):=\{({\bf z},\tau)\in\mathbb R^3: |{\bf z}|\leq r_\ast, 0\leq \tau-\tau_\ast\leq \mathcal T\}$. 
 Let ${\bf z }(t)$ be a solution of system (1) with initial data $|{\bf z}(\tau_\ast)|\leq  \delta_1$ and $\mathfrak T_{\varepsilon}$ be the first exit time of $({\bf z}(\tau), \tau)$ from the domain $\mathfrak D(\varepsilon,\tau_\ast,\mathcal T)$ with some $\delta_1\in (0,\varepsilon_1)$. Define the function $\mathfrak T_{\varepsilon_1,\tau}\equiv \min\{\mathfrak T_{\varepsilon_1}, \tau\}$. Then, ${\bf z}(\mathfrak T_{\varepsilon_1,\tau})$ is the process stopped at the first exit time from the domain $\mathfrak D(\varepsilon_1,\tau_\ast,\mathcal T)$. It follows from \eqref{V1est} that $\mathcal V(z_1(\mathfrak T_{\varepsilon_1,\tau}), z_2(\mathfrak T_{\varepsilon_1,\tau}),\mathfrak T_{\varepsilon_1,\tau})$ is a nonnegative supermartingale (see, for example,~\cite[\S 5.2]{RH12}). In this
case, the following estimates hold:
\begin{align*}
\mathbb P\left(\sup_{0\leq \tau-\tau_\ast\leq \mathcal T} |{\bf z}(\tau)|\geq \varepsilon_1\right) & = 
\mathbb P\left(\sup_{\tau\geq \tau_\ast} |{\bf z}(\mathfrak T_{\varepsilon_1,\tau})|\geq \varepsilon_1\right)  \\
& \leq \mathbb P\left(\sup_{\tau\geq \tau_\ast} \mathcal V(z_1(\mathfrak T_{\varepsilon_1,\tau}), z_2(\mathfrak T_{\varepsilon_1,\tau}),\mathfrak T_{\varepsilon_1,\tau})\geq \mathcal V_-\varepsilon_1^2\right) \\
&\leq 
\frac{ \mathcal V(z_1(\tau_\ast),z_2(\tau_\ast),\tau_\ast)}{\mathcal V_-\varepsilon_1^2}.
\end{align*}
The last estimate follows from Doob's inequality for supermartingales. It follows from \eqref{V0ests} and \eqref{SLF1} that $\mathcal V(z_1(\tau_\ast),z_2(\tau_\ast),\tau_\ast)\leq \mathcal V_+ \delta_1^2+\mu^{2(1-\epsilon)} \delta_2^{2\epsilon} \mathcal V_1(\tau_\ast)$ with some $\epsilon\in(0,1)$.
Hence, taking
\begin{align*}
&\delta_1=\varepsilon_1\sqrt{ \frac{\varepsilon_2 \mathcal V_-}{2\mathcal V_+}}, \\ 
&\delta_2= 
\begin{cases}
\left(\delta_1^2 \mathcal V_+  D_1^{-1}\right)^{\frac{1}{2\epsilon}} \tau_\ast^{-\frac{B-2(C-A)}{2\epsilon B}}, &  \text{if}\quad C \leq A+{B}/{2},\\
\left( \delta_1^2\mathcal V_+ D_1^{-1} \right)^{\frac{1}{2}} \tau_\ast^{\frac{2(C-A)-B}{2B}}, & \text{if}\quad C > A+{B}/{2}.
\end{cases} 
\\
&\mathcal T=
\begin{cases}
 \tau_\ast\left(\mu^{-2(1-\epsilon)}-1\right), & \text{if}\quad  C < A+{B}/{2},\\
\tau_\ast \left(\exp{ \mu^{-2(1-\epsilon)}}-1\right), & \text{if}\quad C  = A+{B}/{2},\\
\infty, & \text{if}\quad  C  > A+{B}/{2}.
\end{cases}
\end{align*}
yields
\begin{gather*}
\mathbb P\left(\sup_{0\leq \tau-\tau_\ast\leq \mathcal T} |{\bf z}(\tau)|\geq \varepsilon_1\right) \leq \varepsilon_2.
\end{gather*}
Returning to the variables $\rho(t)$, $\varphi(t)$, we obtain \eqref{Th2est} with $t_\ast=(B\tau_\ast)^{1/B}$ and $\tilde{\mathcal T}=(B \mathcal T+t_\ast^B)^{1/B}-t_\ast$.
\end{proof}
 
\section{Examples}
\label{Sex}
Consider again equation \eqref{Ex}. In this case, we have
\begin{gather*}
h=1,\quad l=m=0,\quad a=p-2,\quad b=n-2,\\
\mathcal M_1=-\alpha+(p-2)\beta, \quad \mathcal M_2=2(-\gamma+(n-2)\beta).
\end{gather*}
It follows from~\cite{IKetal16} that the solution of system \eqref{Xi0} has the form
\begin{gather*}
X_{1,0}(\varphi)\equiv \sqrt 2 {\hbox{\rm cn}}\left(\frac{T_0\varphi}{\pi\sqrt 2};\frac{1}{2}\right), \quad 
X_{2,0}(\varphi)\equiv \nu_0\partial_\varphi X_{1,0}(\varphi)
\end{gather*}
with $T_0=2\sqrt{2}K(1/2)$ and $\nu_0=2\pi/T_0$, where $K(k)$ is a complete elliptic integral of the first kind and ${\hbox{\rm cn}}(\phi;k)$ is the Jacobi elliptic function. Moreover, it follows from~\cite[\S 63]{NIA90} that $2\pi$-periodic function $X_{1,0}(\varphi)$ admits the Fourier expansion
\begin{gather*}
X_{1,0}(\varphi)=\sum_{k=1}^\infty q_k \cos\big((2k-1) \varphi\big), \quad q_k=2 \nu_0\sqrt 2\,{\hbox {\rm sech}}\left((2k-1)\frac{\pi}{2}\right).
\end{gather*}
The conditions \eqref{ncond} and \eqref{adas} correspond to  the inequalities
\begin{gather}\label{ineqex}
2-\frac{1}{\beta}\leq p-\frac{\alpha}{\beta} <3, \quad 
n-\frac{\gamma}{\beta}\leq \frac{3}{4}\left(p-\frac{\alpha}{\beta}+\frac{1}{3}\right).
\end{gather}

{\bf 1}. Let $p=n=0$, $\alpha=\beta=1/3$, and $\gamma=1/6$. In this case, inequalities \eqref{ineqex} are satisfied and system \eqref{Ex} takes the form
\begin{gather}\label{Ex1}
dx_1=x_2\,dt, \quad dx_2=\left(x_1-x_1^3+t^{-\frac{1}{3}}  \mathcal Q  \cos \left(s t^{\frac{4}{3}}\right)\right) dt+ t^{-\frac{1}{6}} \mu   \cos\left(s t^{\frac{4}{3}}\right)\, dw(t).
\end{gather}
From \eqref{ABC}, \eqref{Pdef} and \eqref{Jdef} it follows that $a=b=-2$, $A=2/3$, $B=4/3$, $C=1$, and
\begin{align*}
\mathcal P(\Theta) & \equiv \frac{1}{2\sqrt 3}\Big( \frac{3}{z_0^2} \langle  f_{1,0}(\varkappa^{-1}\zeta+\Theta,\zeta) \rangle_{\varkappa\zeta}-1\Big),\\
J(\Theta) &  \equiv \frac{3}{ 4 z_0^2}\left\langle \partial_\Theta f_{2,0}(\varkappa^{-1}\zeta+\Theta,\zeta)  -f_{1,0}(\varkappa^{-1}\zeta+\Theta,\zeta) \right\rangle_{\varkappa\zeta}+\frac{1}{4},
\end{align*}
where
\begin{gather*}
f_{1,0}(\varphi,S)\equiv \frac{ \mathcal Q \nu_0}{4} \partial_\varphi X_{1,0}(\varphi) \cos S , \quad 
f_{2,0}(\varphi,S)\equiv -\frac{ \mathcal Q \nu_0}{4} X_{1,0}(\varphi)  \cos S, \quad 
z_0=\frac{4 s}{3 \nu_0\varkappa }.
\end{gather*}
Note that 
\begin{gather*}
\langle Z(\varkappa^{-1}\zeta+\Theta)\cos\zeta\rangle_{\varkappa \zeta}\equiv
\langle Z(\zeta)\cos\varkappa\zeta\rangle_{\zeta}\cos\varkappa\Theta +
\langle Z(\zeta)\sin\varkappa\zeta\rangle_{\zeta}\sin\varkappa\Theta 
\end{gather*}
for any $2\pi$-periodic function $Z(\varphi)$. In the case of a  resonance with $\varkappa=2k-1$, $k\in\mathbb Z_+$, we have
\begin{align*}
\mathcal P(\Theta)&\equiv -\frac{1}{2\sqrt 3}\left(\frac{\mathcal Q }{  \mathcal Q_\varkappa s^2} \sin \big((2k-1)\Theta\big)+1\right), \\ 
J(\Theta)&\equiv \frac{  \mathcal Q }{2  \mathcal Q_\varkappa s^2 }  \sin\big((2k-1)\Theta\big)+\frac{1}{4},\\
\mathcal Q_\varkappa&=\frac{128}{(3\nu_0 (2k-1))^3q_k}.
\end{align*} 
Therefore, if $|\mathcal Q|/s^2> \mathcal  Q_\varkappa$, then the condition \eqref{astheta} holds with
\begin{gather*}
\Theta_0\in\left \{(-1)^j\theta_0+\frac{\pi j}{2k-1}\right\}_{ j\in\mathbb Z}, \quad 
\theta_0=\frac{1}{2k-1}\arcsin\left(-\frac{ \mathcal Q_\varkappa s^2}{\mathcal Q}\right).
\end{gather*}
It can easily be checked that for all $j\in\mathbb Z$
\begin{gather*}
\mathcal P'(\Theta_0)=-\frac{(-1)^j\mathcal Q  (2k-1) }{2\sqrt 3 \mathcal Q_\varkappa s^2}\cos\big((2k-1)\theta_0\big)\neq 0, \quad J(\Theta_0)=-\frac{1}{4}.
\end{gather*}
Thus, if $|\mathcal Q|/s^{2}>\mathcal Q_\varkappa$ and $(-1)^j \mathcal Q > 0$, then it follows from Lemma~\ref{LTh2} and Theorem~\ref{Th2} that a phase-locking regime occurs in system \eqref{Ex1} such that
\begin{align*}
\rho(t)& \approx  \frac{4 s t^\frac{1}{3}}{3 \nu_0 (2k-1) }, \\ 
\varphi(t)& \approx  \frac{1}{2k-1}\left(S(t)+(-1)^j\arcsin\left(-\frac{\mathcal Q_\varkappa s^2}{\mathcal Q}\right)+ \pi j\right)
\end{align*}
for all $k\in\mathbb Z_+$, where $\rho(t)\equiv [H(x_1(t),x_2(t))]^{1/4}$ and $\varphi(t)\equiv Y_2(x(t),\dot x(t))$ correspond to the amplitude and the phase of solutions to system \eqref{Ex1}.
Moreover, since $C<A+B/2$, we see that this regime is stochastically stable on an asymptotically large time interval $t\leq \mathcal O (\mu^{-1+\epsilon})$ as $\mu\to 0$ for all $\epsilon\in(0,1)$ (see~Fig.~\ref{fex1}). Note that if $|\mathcal Q|/s^2>\mathcal Q_\varkappa$ and $(-1)^j \mathcal Q<0$, this regime is unstable in the corresponding truncated system. Finally, it follows from Lemma~\ref{LTh3} that the capture into resonance does not occur if $|\mathcal Q|/s^2<\mathcal Q_\varkappa$. 

\begin{figure}
\centering
 \includegraphics[width=0.4\linewidth]{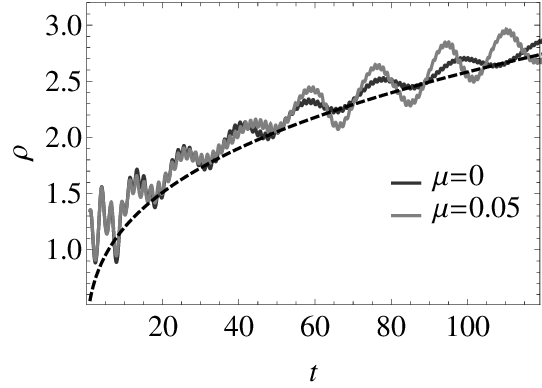}
\hspace{1ex}
  \includegraphics[width=0.4\linewidth]{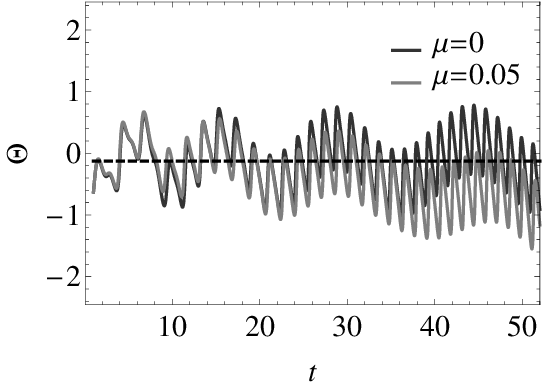}
\caption{\small The evolution of $\rho(t)$ and $\Theta(t)=\varphi(t)-\varkappa^{-1}S(t)$ for sample paths of the resonant solutions $(\varkappa=1)$ to system \eqref{Ex1} with $\mathcal Q=4$ and $s=1/2$. The black dashed curves correspond to $\rho(t)\equiv z_0 t^{1/3}$ and $\Theta(t)\equiv \theta_0$, where $z_0\approx 0.56$, $\theta_0\approx -0.127$. In this case, $\mathcal Q_1\approx 2.04$.} \label{fex1}
\end{figure}

{\bf 2}. Now let $p=n=1$, $\alpha=\beta=1/2$, and $\gamma=3/4$.
Then, system \eqref{Ex} has the following form:
\begin{gather}\label{Ex2}
dx_1=x_2\,dt, \quad dx_2=\left(x_1-x_1^3+t^{-\frac{1}{2}}  \mathcal Q  x_1 \cos \left(s t^{\frac{3}{2}}\right)\right) dt+ t^{-\frac{3}{4}} \mu  x_1 \cos \left(s t^{\frac{3}{2}}\right)\, dw(t).
\end{gather}
We see that the inequalities \eqref{ineqex} are satisfied. In this case, $a=b=-1$, $A=3/4$, $B=3/2$, $C=3/2$,
\begin{align*}
\mathcal P(\Theta) & \equiv \frac{1}{\sqrt 6}\left( \frac{2}{z_0} \langle  f_{1,0}(\varkappa^{-1}\zeta+\Theta,\zeta) \rangle_{\varkappa\zeta}-1\right),\\
J(\Theta) &  \equiv \frac{2}{ 3 z_0}\left\langle \partial_\Theta f_{2,0}(\varkappa^{-1}\zeta+\Theta,\zeta) \right\rangle_{\varkappa\zeta}+\frac{1}{6},
\end{align*}
where
\begin{gather*}
f_{1,0}(\varphi,S)\equiv \frac{ \mathcal Q \nu_0}{8} \partial_\varphi(X_{1,0}(\varphi))^2 \cos S , \quad 
f_{2,0}(\varphi,S)\equiv -\frac{ \mathcal Q \nu_0}{4} (X_{1,0}(\varphi))^2  \cos S, \quad 
z_0=\frac{3s}{2\nu_0\varkappa}.
\end{gather*}
 In the case of a resonance with $\varkappa=2k$, $k\in\mathbb Z_+$, we have
\begin{gather}\label{PJ}
\mathcal P(\Theta)\equiv -\frac{1}{\sqrt 6}\left(\frac{\mathcal Q }{s \mathcal Q_k } \sin (2k\Theta)+1\right), \quad 
J(\Theta)\equiv \frac{ 2 \mathcal Q }{ s\mathcal Q_\varkappa}  \sin(2k\Theta)+\frac{1}{6},
\quad 
\mathcal Q_\varkappa=\frac{3}{2(\nu_0 k)^2 \tilde q_k},
\end{gather} 
where $\tilde q_k=\langle (X_{1,0}(\zeta))^2\cos (2k\zeta)\rangle_{\zeta}$.
It follows that if $|\mathcal Q|/s> \mathcal  Q_\varkappa$, then the condition \eqref{astheta} holds with
\begin{gather*}
\Theta_0\in\left \{(-1)^j\theta_0+\frac{\pi j}{2k}\right\}_{ j\in\mathbb Z}, \quad 
\theta_0=\frac{1}{2k}\arcsin\left(-\frac{ s \mathcal Q_\varkappa  }{\mathcal Q}\right).
\end{gather*}
It can easily be checked that for all $j\in\mathbb Z$
\begin{gather*}
\mathcal P'(\Theta_0)=-\frac{(-1)^j2k \mathcal Q }{\sqrt 6  s \mathcal Q_\varkappa }\cos(2k\theta_0)\neq 0, \quad J(\Theta_0)=-\frac{1}{2}.
\end{gather*}
It follows from Lemma~\ref{LTh2} and Theorem~\ref{Th2} that if $|\mathcal Q|/s>\mathcal Q_\varkappa$ and $(-1)^j \mathcal Q > 0$, then a phase-locking regime occurs in system \eqref{Ex2} such that
\begin{gather*}
\rho(t)\approx   \frac{3s t^{\frac 12}}{4\nu_0 k}  , \quad \varphi(t)\approx  \frac{1}{2k}\left(S(t)+(-1)^j\arcsin\left(-\frac{s\mathcal Q_\varkappa}{\mathcal Q}\right)+\pi j\right), \quad k\in\mathbb Z_+.
\end{gather*}
Moreover, since $C=A+B/2$, we see that this mode is stochastically stable on the exponentially long time interval $t\leq \mathcal O (\exp \mu^{-4(1-\epsilon)/3})$ as $\mu\to 0$ for all $\epsilon\in(0,1)$ (see~Fig.~\ref{fex2}). If $|\mathcal Q|/s>\mathcal Q_\varkappa$ and $(-1)^j \mathcal Q<0$, this regime is unstable. The capture into resonance does not occur if $|\mathcal Q|/s<\mathcal Q_\varkappa$. 
\begin{figure}
\centering
 \includegraphics[width=0.4\linewidth]{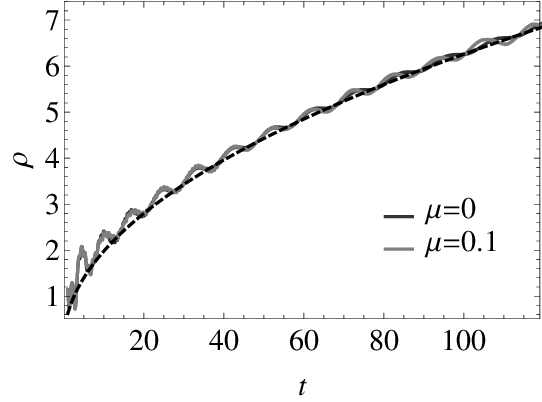}
\hspace{1ex}
  \includegraphics[width=0.4\linewidth]{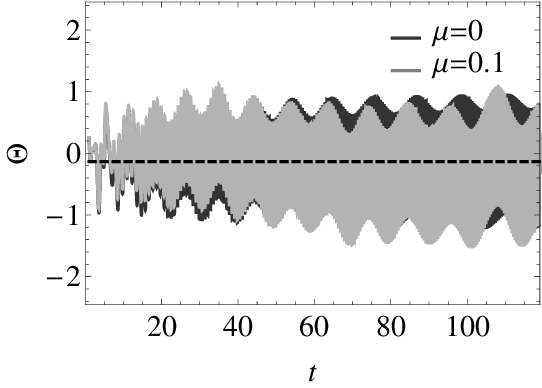}
\caption{\small The evolution of $\rho(t)$ and $\Theta(t)=\varphi(t)-\varkappa^{-1}S(t)$ for sample paths of the resonant solutions $(\varkappa=2)$ to system \eqref{Ex2} with $\mathcal Q=8$ and $s=1$. The black dashed curves correspond to $\rho(t)\equiv z_0 t^{1/2}$ and $\Theta(t)\equiv \theta_0$, where $z_0\approx 0.63$, $\theta_0\approx -0.132$. In this case, $\mathcal Q_2\approx 2.101$.} \label{fex2}
\end{figure}

{\bf 3}. Finally, let $p=1$, $n=2$, $\alpha=\beta=1/2$, and $\gamma=3/2$.
Then, system \eqref{Ex} takes the form
\begin{gather}\label{Ex3}
dx_1=x_2\,dt, \quad dx_2=\left(x_1-x_1^3+t^{-\frac{1}{2}}  \mathcal Q  x_1 \cos \left(s t^{\frac{3}{2}}\right)\right) dt+ t^{-\frac{3}{2}} \mu  x_1^2 \cos \left(s t^{\frac{3}{2}}\right)\, dw(t).
\end{gather}
It can easily be checked that the inequalities \eqref{ineqex} hold. It follows from \eqref{ABC}, \eqref{Pdef} and \eqref{Jdef} that $a=b=0$, $A=3/4$, $B=3/2$, $C=7/4$. Note that in the case of $\varkappa=2k$, $k\in\mathbb Z_+$, the functions $\mathcal P(\Theta)$ and $J(\Theta)$ have the form \eqref{PJ}. Thus, the results obtained for the previous example are valid.  Moreover, since $C>A+B/2$, we obtain the stochastic stability of the resonance on an semi-infinite time interval if $|\mathcal Q|>\mathcal Q_\varkappa$ and $(-1)^j \mathcal Q>0$ (see~Fig.~\ref{fex3}).   

\begin{figure}
\centering
 \includegraphics[width=0.4\linewidth]{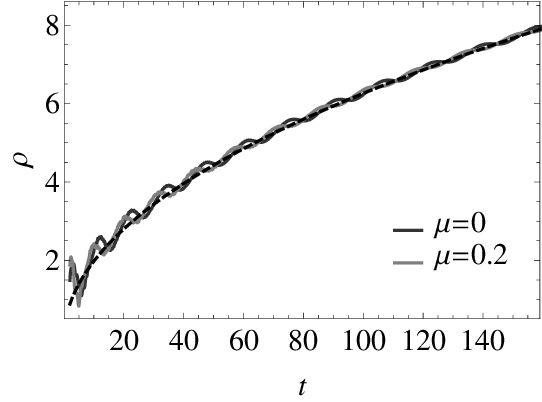}
\hspace{1ex}
  \includegraphics[width=0.4\linewidth]{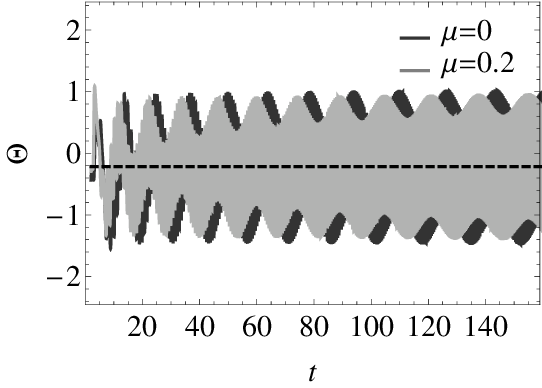}
\caption{\small The evolution of $\rho(t)$ and $\Theta(t)=\varphi(t)-\varkappa^{-1}S(t)$ for sample paths of the resonant solutions $(\varkappa=2)$ to system \eqref{Ex3} with $\mathcal Q=5$ and $s=1$. The black dashed curves correspond to $\rho(t)\equiv z_0 t^{1/2}$ and $\Theta(t)\equiv \theta_0$, where $z_0\approx 0.63$, $\theta_0\approx -0.132$. In this case, $\mathcal Q_2\approx 2.101$.} \label{fex3}
\end{figure}

\section{Conclusion}
Thus, the combined influence of a decaying chirped-frequency driving and a stochastic perturbation on the class of nonlinear systems far from the equilibrium have been investigated. Using a modified averaging method, we have derived the model truncated system \eqref{det1}, describing a possible dynamics in the perturbed system \eqref{ps}. We have shown that the model system has at least two regimes: phase locking and phase drifting. The resonant solutions with unlimitedly growing amplitude arise in the phase-locking mode. We have described the conditions that guarantee the persistent of such solutions in the full stochastic system on asymptotically large time intervals.  

Note that the proposed technique cannot be applied directly to the case of more complicated limiting systems than \eqref{limsys}. The same is true for multidimensional systems with chirped-frequency perturbations due to the small denominators that may appear when averaging the equations. These cases deserve special attention and will be discussed elsewhere.

\section*{Acknowledgments}
The work is supported by the Russian Science Foundation (project No. 19-71-30002).

\end{document}